\documentclass{amsart}
\usepackage{amsmath,amssymb,amscd,amsthm,amsfonts,amstext,amsbsy,mathrsfs,hyperref,upgreek,mathtools,stmaryrd,enumitem,bbm}
\usepackage{verbatim}

\usepackage[textsize=footnotesize,color=blue!40, bordercolor=white]{todonotes}

\usepackage
[a4paper, margin=1.2in]{geometry}

\newcommand\Th{\mathsf{T}}
\newcommand\V{\mathsf{V}}

\renewcommand{\L}{\mathcal{L}}
\newcommand{\C}{\mathcal{C}}

\newcommand{\Ss}{\mathbb{S}}

\newcommand{\PP}{\mathbb{P}}
\newcommand{\QQ}{\mathbb{Q}}

\newcommand{\BB}{\mathbb{B}}
\newcommand{\MM}{\mathbb{M}}

\newcommand{\FF}{\mathbb{F}}

\newcommand{\anf}[1]{{\text{``}\hspace{0.3ex}{#1}\hspace{0.01ex}\text{''}}}

\newcommand{\range}{\operatorname{range}}
\newcommand{\rank}{\operatorname{rank}}
\newcommand{\Col}{\operatorname{Col}}
\newcommand{\ra}{\rightarrow}
\newcommand{\lr}{\leftrightarrow}
\newcommand{\Llr}{\Longleftrightarrow}
\newcommand{\tc}{\operatorname{tc}}

\newcommand{\vn}{\vec{n}}
\newcommand{\vm}{\vec{m}}
\newcommand{\one}{\mathbbm{1}}
\newcommand{\op}{\mathsf{op}}

\newcommand{\rnk}{\operatorname{rnk}}

\newcommand{\On}{{\mathrm{Ord}}}

\newcommand{\Fm}{\mathrm{Fml}}

\newcommand{\ZFC}{{\sf ZFC}}
\newcommand{\ZF}{{\sf ZF}}

\newcommand{\KM}{{\sf KM}}
\newcommand{\GCH}{{\sf GCH}}

\newcommand{\GB}{{\sf GB}}
\newcommand{\GBC}{{\sf GBC}}

\newcommand{\pmap}[4]{{#1}:{#2}\xrightarrow{#4}{#3}}
\newcommand{\Set}[2]{\{{#1}~\vert~{#2}\}}
\newcommand{\ran}[1]{{{\rm{ran}}(#1)}}
\newcommand{\dom}[1]{{{\rm{dom}}(#1)}}

\newcommand{\gbr}[1]{{\ulcorner{#1}\urcorner}}

\newtheorem{theorem}{Theorem}[section]
\newtheorem{lemma}[theorem]{Lemma}
\newtheorem{corollary}[theorem]{Corollary}

\newtheorem{question}[theorem]{Question}

\newtheorem*{claim*}{Claim}
\newtheorem*{subclaim*}{Subclaim}

\theoremstyle{definition}
\newtheorem{definition}[theorem]{Definition}

\newtheorem{example}[theorem]{Example}
\theoremstyle{remark}
\newtheorem{remark}[theorem]{Remark}
\newtheorem*{notation}{Notation}

\newenvironment{enumerate-(a)}{\begin{enumerate}[label={\upshape (\alph*)}, leftmargin=2pc]}{\end{enumerate}}

\newenvironment{enumerate-(a)-r}{\begin{enumerate}[label={\upshape (\alph*)}, leftmargin=2pc,resume]}{\end{enumerate}}

\newenvironment{enumerate-(A)}{\begin{enumerate}[label={\upshape (\Alph*)}, leftmargin=2pc]}{\end{enumerate}}

\newenvironment{enumerate-(A)-r}{\begin{enumerate}[label={\upshape (\Alph*)}, leftmargin=2pc,resume]}{\end{enumerate}}

\newenvironment{enumerate-(i)}{\begin{enumerate}[label={\upshape (\roman*)}, leftmargin=2pc]}{\end{enumerate}}

\newenvironment{enumerate-(i)-r}{\begin{enumerate}[label={\upshape (\roman*)}, leftmargin=2pc,resume]}{\end{enumerate}}

\newenvironment{enumerate-(I)}{\begin{enumerate}[label={\upshape (\Roman*)}, leftmargin=2pc]}{\end{enumerate}}

\newenvironment{enumerate-(I)-r}{\begin{enumerate}[label={\upshape (\Roman*)}, leftmargin=2pc,resume]}{\end{enumerate}}

\newenvironment{enumerate-(1)}{\begin{enumerate}[label={\upshape (\arabic*)}, leftmargin=2pc]}{\end{enumerate}}

\newenvironment{enumerate-(1)-r}{\begin{enumerate}[label={\upshape (\arabic*)}, leftmargin=2pc,resume]}{\end{enumerate}}

\begin{document}

\thanks{The first, third and fifth author were partially supported by DFG-grant LU2020/1-1. The authors would like to thank Joel David Hamkins for the careful reading of the paper and numerous useful comments.}



\author{Peter Holy}
\address{Peter Holy, Math. Institut, Universit\"at Bonn,
Endenicher Allee 60, 53115 Bonn, Germany}
\email{pholy@math.uni-bonn.de}
\urladdr{}

\author{Regula Krapf}
\address{Regula Krapf, Math. Institut, Universit\"at Bonn,
Endenicher Allee 60, 53115 Bonn, Germany}
\email{krapf@math.uni-bonn.de}
\urladdr{}

\author{Philipp L\"ucke}
\address{Philipp L\"ucke, Math. Institut, Universit\"at Bonn,
Endenicher Allee 60, 53115 Bonn, Germany}
\email{pluecke@math.uni-bonn.de}
\urladdr{}

\author{Ana Njegomir}
\address{Ana Njegomir, Math. Institut, Universit\"at Bonn,
Endenicher Allee 60, 53115 Bonn, Germany}
\email{njegomir@math.uni-bonn.de}
\urladdr{}

\author{Philipp Schlicht}
\address{Philipp Schlicht, Math. Institut, Universit\"at Bonn, 
Endenicher Allee 60, 53115 Bonn, Germany and 
Institut f\"ur Mathematische Logik und Grundlagenforschung, Universit\"at M\"unster, 
Einsteinstr. 62, 48149 M\"unster, Germany}
\email{schlicht@math.uni-bonn.de}
\urladdr{}

\title{Class forcing, the forcing theorem and Boolean completions}

\begin{abstract} 
The forcing theorem is the most fundamental result about set forcing, stating that the forcing relation for any set forcing is definable and that the truth lemma holds, that is everything that holds in a generic extension is forced by a condition in the relevant generic filter. We show that both the definability (and, in fact, even the amenability) of the forcing relation and the truth lemma can fail for class forcing.

In addition to these negative results, we show that the forcing theorem is equivalent to the existence of a (certain kind of) Boolean completion, and we introduce a weak combinatorial property (\emph{approachability by projections}) that implies the forcing theorem to hold. Finally, we show that unlike for set forcing, Boolean completions need not be unique for class forcing. 
\end{abstract}

\maketitle 

\setcounter{tocdepth}{2}

\section{Introduction}\label{section:Intro}

The classical approach to generalize the technique of forcing with set-sized partial orders to forcing with class partial orders is to work with countable 
transitive models $M$ of some  theory extending $\ZF^-$ and 
partial orders $\PP$ definable over $M$ (in the sense that both the domain of $\PP$ and the relation $\leq_\PP$ 
are definable over the model $\langle M,\in\rangle$). By $\ZF^-$ we mean the usual axioms of $\ZF$ without the power set axiom,
however including Collection instead of Replacement.\footnote{Note that in the absence of the power set axiom, Collection does not follow from Replacement and 
many important set-theoretical results can consistently fail in the weaker theory. For further details, consult \cite{ZFCminus}.} In this situation, we say that a filter $G$ on $\PP$ is 
$\PP$-generic over $M$ if $G$ meets every dense subset of $\PP$ that is definable over $M$. We let $M^\PP$ denote the collection 
of all $\PP$-names contained in $M$. 
Since $M\models\ZF^-$, $M^\PP$ is definable over $M$. Given a $\PP$-generic filter $G$ over $M$, we define $M[G]=\Set{\sigma^G}{\sigma\in M^\PP}$ 
to be the corresponding class generic extension. 
Finally, given a formula 
$\varphi(v_0,\ldots,v_{n-1})$ in the language $\mathcal{L}_\in$ of set theory, a 
condition $p$ in $\PP$ and $\sigma_0,\ldots,\sigma_{n-1}\in M^\PP$, we let 
$p\Vdash^M_\PP\varphi(\sigma_0,\ldots,\sigma_{n-1})$ denote the statement that $\varphi(\sigma_0^G,\ldots,\sigma_{n-1}^G)$ holds in $M[G]$ whenever $G$ is a $\PP$-generic filter over $M$ with $p\in G$. 

The \emph{forcing theorem} 
is the most fundamental result in the theory of forcing with set-sized partial orders. 
The work presented in this paper is motivated by the question whether fragments of this result also hold for class forcing. Given a countable transitive model $M$ of some theory extending $\ZF^-$, a partial order $\PP$ definable over $M$ and an $\mathcal{L}_\in$-formula $\varphi(v_0,\ldots,v_{n-1})$, we will consider the following fragments of the forcing theorem for notions of class forcing.

\begin{enumerate}
 \item We say that \emph{$\PP$ satisfies the definability lemma for $\varphi$ over $M$}  if the set $$\Set{\langle p,\sigma_0,\ldots,\sigma_{n-1}\rangle\in\PP\times M^\PP\times\ldots\times M^\PP}{p\Vdash^M_\PP\varphi(\sigma_0,\ldots,\sigma_{n-1})}$$ is definable over $M$. 

 \item We say that \emph{$\PP$ satisfies the truth lemma for $\varphi$ over $M$} if for all $\sigma_0,\ldots,\sigma_{n-1}\in M^\PP$ and every $\PP$-generic filter $G$ over $M$ with the property that $\varphi(\sigma_0^G,\ldots,\sigma_{n-1}^G)$ holds in $M[G]$, there is a condition $p\in G$ with $p\Vdash^M_\PP\varphi(\sigma_0,\ldots,\sigma_{n-1})$. 

 \item We say that \emph{$\PP$ satisfies the forcing theorem for $\varphi$ over $M$} if $\PP$ satisfies both the definability and the truth lemma for $\varphi$ over $M$. 
\end{enumerate}

Another basic result that is fundamental for the development of forcing with set-sized partial orders is the existence of a Boolean completion for separative partial orders,\footnote{A partial order (or, more generally, a preorder) $P$ is \emph{separative} if for all $p,q\in P$, if $p\not\le q$ then there exists $r\le p$ such that $r\perp q$.} its uniqueness up to isomorphism and the equality of the corresponding forcing extensions. This motivates our interest in the following notions.

\begin{enumerate}
 \setcounter{enumi}{3}

 \item Let $\BB$ be a Boolean algebra that is definable over $M$ (in the sense that both the domain of $\BB$ and all Boolean operations are definable over the model $\langle M,\in\rangle$). We say that $\BB$ is \emph{$M$-complete} if $\sup_\BB A$ exists in $\BB$ for every $A\subseteq\BB$ with $A\in M$. 

 \item We say that \emph{$\PP$ has a Boolean completion in $M$} if there is an $M$-complete Boolean algebra $\BB$ and an injective dense embedding $\pi$ from $\PP$ into $\BB\setminus\{0_\BB\}$ such that both $\BB$ and $\pi$ are definable over $M$.

 \item We say that \emph{$\PP$ has a unique Boolean completion in $M$} if $\PP$ has a Boolean completion $\BB_0$ in $M$ and for every other Boolean completion $\BB_1$ of $\PP$ in $M$, there is an isomorphism in $\V$ between $\BB_0$ and $\BB_1$ which fixes $\PP$. 
\end{enumerate}

In standard accounts on class forcing (see \cite{MR1780138}) one studies generic extensions with additional predicates
for the generic filter and the ground model and focuses on \emph{pretame} (resp.\ \emph{tame}) notions of forcing, i.e.\ notions of forcing which preserve $\ZF^-$ (resp. $\ZF$)
with respect to these predicates. In particular, Sy Friedman shows in \cite{MR1780138} that if the ground model satisfies $\ZF$, then every pretame forcing satisfies the forcing theorem. The converse is false:
A simple notion of class forcing which does not preserve Replacement is $\Col(\omega,\On)$, the class
of all finite partial functions from $\omega$ to $\On$, ordered by reverse inclusion. However, this notion of forcing still satisfies the forcing theorem (see \cite[Proposition 2.25]{MR1780138} or Section \ref{section:Approach} of this paper). In this paper, we will mostly investigate properties of non-pretame notions of class forcing.


In the remainder of this introduction, we present the results of this paper. We will later prove these statements in a more 
general setting than the one outlined above. This will allow us to also prove results for models containing more second-order 
objects, like models of \emph{Kelley-Morse class theory $\KM$} (see \cite{antos2015class}). We will outline this setting in Section \ref{section:general setting}.

Throughout this paper, we will work in a model $\V$ of $\ZFC$ and, given a set $M$ and a recursively enumerable theory $\Th$ extending $\ZF^-$, we say that ``\emph{$M$ is a model of $\Th$}'' to abbreviate the statement that $M$ satisfies every axiom of $\Th$ in $\V$ with respect to some formalized satisfaction relation (as in {\cite[Chapter 3.5]{drake1974set}}). Note that, in general, the assumption that such a model $\V$ containing a transitive countable set $M$ of $\Th$ exists is stronger than the assumption that $\Th$ is consistent. If $\Th\supseteq\ZF$, then the results of this paper can also be proven in the setting of {\cite[Ch.\ VII, \S 9, Approach (1b)]{MR597342}}, where one works with a language that extends the language of set theory by a constant symbol $\dot{M}$, and a model of a theory in this language that extends $\Th$ by the scheme of axioms stating that every axiom of $\Th$ holds relativized to $\dot{M}$. The consistency of this theory is equivalent to the consistency of $\Th$. For this 
paper, we have nonetheless chosen the first approach, because it makes many arguments more intuitive and easier to state.


\subsection{Positive results}

The results of this paper will show that the forcing theorem (in fact both the definability of the forcing
relation and the truth lemma), the amenability of the forcing relation and the existence of a 
(unique) Boolean completion can fail for class 
forcing. The following two positive results show that non-trivial implications hold between some of these properties. The first result shows that a failure of the forcing theorem already implies a failure of the definability lemma for atomic formulae. This result is proven by carefully mimicking the induction steps in the proof of the forcing theorem for set forcing.

\begin{theorem}\label{theorem:ForcingTheorem}
 Let $M$ be a countable transitive model of $\ZF^-$ and let $\PP$ be a partial order that is definable over $M$. 
 If $\PP$ satisfies the definability lemma for either \anf{$v_0\in v_1$} or \anf{$v_0=v_1$} over $M$, then $\PP$ satisfies the 
 forcing theorem for all $\mathcal{L}_\in$-formulae over $M$. 
\end{theorem}

In Section \ref{section:Approach}, we will present a criterion that will allow us to show that many notions of class forcing satisfy the definability lemma for atomic formulae and thus, by the above result, the full forcing theorem (that is, the forcing theorem for all $\mathcal{L}_\in$-formulae).

The next result shows that the existence of a Boolean completion is equivalent to the validity of the forcing theorem for all $\mathcal{L}_\in$-formulae. We will prove this result by showing that the forcing relation for the quantifier-free infinitary language $\mathcal{L}_{\On,0}$ of set theory, allowing set-sized conjunctions and disjunctions and also allowing reference to a predicate for the generic filter, is definable under either assumption listed in the theorem.

\begin{theorem}\label{theorem:BooleanCompletion}
 Let $M$ be a countable transitive model of $\ZF^-$ and let $\PP$ be a separative partial order that is definable over $M$. If either the power set axiom holds in $M$ or there is a well-ordering of $M$ that is definable over $M$, then the following statements are equivalent.
 \begin{enumerate}
  \item $\PP$ satisfies the forcing theorem for all $\mathcal{L}_\in$-formulae over $M$. 
  \item $\PP$ has a Boolean completion in $M$. 
 \end{enumerate}
\end{theorem}


\subsection{Negative results}

In the following, we present results showing that each of the properties considered above can fail for class forcing. The first result shows that there is always a notion of class forcing that does not satisfy the definability lemma. The proof of this result uses a notion of class forcing that was introduced by Sy Friedman and that is mentioned in \cite[Remark 1.8]{MR1976584}. We will present and study this notion of forcing in detail in Section \ref{examples}.

\begin{theorem}\label{theorem:FailureDefLemma}
 Let $M$ be a countable transitive model of $\ZF^-$.  Then there is a partial order $\PP$ that is definable over $M$ and does not satisfy the forcing theorem for atomic formulae over $M$. 
\end{theorem}

Our next result shows that even stronger failures of the definability lemma are possible for the above forcing. 
Its proof relies on so-called \emph{Paris models}, i.e.\ $\in$-structures $\mathcal{M}$ with the 
property that each ordinal of $\mathcal M$ is definable in $\mathcal{M}$ by a formula without parameters. Such models have been
considered by Ali Enayat in \cite{MR2140616}.
The stronger concept of \emph{pointwise definable models} (that is, $\in$-structures over which each of their 
elements is definable by a formula without parameters) was studied in depth in \cite{MR3087066}. Note that the existence 
of a countable transitive model of $\ZFC$ yields the existence of a countable transitive Paris model satisfying the axioms of $\ZFC$
-- this follows from \cite[Theorem 11]{MR3087066}, where it is shown that every countable 
transitive model of $\ZFC$ has a pointwise definable class forcing extension. However, in Section \ref{sec:ftl} we will, 
for the benefit of the reader, sketch a simplified argument to verify (the weaker statement) that certain countable 
transitive models of $\ZFC$ have class forcing extensions which are Paris models. 


\begin{theorem}\label{theorem:FailureAmenability}
Let $M$ be a countable transitive Paris model with $M\models\ZF^-$. Then there is a partial order $\PP$ such that 
$\PP$ is definable over $M$ and the $\PP$-forcing relation for $\anf{v_0=v_1}$ is not $M$-amenable, i.e.
there is a set $x\in M$ such that 
$$\{\langle p,\sigma,\tau\rangle\in\PP\times M^\PP\times M^\PP\mid p\Vdash_\PP^M\sigma=\tau\}\cap x$$
is not an element of $M$. 
\end{theorem}


Next, we consider failures of the truth lemma. 
The witnessing forcing notion for the next theorem will be the two-step iteration of the above notion of class forcing that has a generic 
extension which is a Paris model and the aforementioned notion of class forcing of Sy Friedman.

\begin{theorem}\label{theorem:FailureTruthLemma}
Assume that $M$ is a countable transitive model of $\ZF^-$ such that either $M$ is uncountable in $\mathsf L[M]$ and ${}^{\omega}M\cap \mathsf{L}[M] \subseteq M$, or $M$ satisfies $\V=\mathsf L$
and there is a countable subset $\C$ of $\mathcal P(M)$ such that $\langle M,\C\rangle$ is a model of $\KM$. Then there is 
a partial order $\PP$ that is definable over $M$ and that does not satisfy the truth lemma for $\anf{v_0=v_1}$ over $M$. 
\end{theorem}


Finally, we show that the existence and uniqueness of a Boolean completion, in a countable transitive $M\models\ZF^-$ of which there exists a well-order of order-type $\On^M$ that is definable over $M$, of a notion of forcing $\PP$ that is definable over $M$, are equivalent to $\PP$ having the \emph{$\On$-chain condition} (or simply \emph{$\On$-cc}) over $M$, that is the property that every antichain of $\PP$ that is definable over $M$ is already an element of $M$. This will easily yield the following result. 

\begin{theorem}\label{theorem:NotUniqueBoolCompl}
Let $M$ be a countable transitive model of $\ZF^-$ and suppose that there exists a global well-order of order type $\On^M$ that is definable over $M$.
Then there is a notion of class forcing which has two non-isomorphic Boolean completions in $M$. 
\end{theorem}


\section{Some notions of class forcing}\label{examples}

In this section, we introduce several notions of class forcing that will later be used to verify the negative results listed in Section \ref{section:Intro}. 

\begin{notation}
Since we will frequently use names for ordered pairs, we introduce the notation
\begin{align*}
\op(\sigma,\tau)&=\{\langle\{\langle\sigma,\one_\PP\rangle\},\one_\PP\rangle,\langle\{\langle\sigma,\one_\PP\rangle,\langle\tau,\one_\PP\rangle\},\one_\PP\rangle\}
\end{align*}
for $\sigma,\tau\in M^\PP$ and $\alpha\in\On$. Clearly, $\op(\sigma,\tau)$ is the canonical 
name for the ordered pair $\langle\sigma^G,\tau^G\rangle$.
\end{notation}

\begin{definition}\label{def:Col}
 Let $M$ be a countable transitive model of $\ZF^-$.
\begin{enumerate}

 \item Let $\Col(\omega,\On)^M$ denote the partial order $\Col(\omega,\On^M)$, i.e. $\Col(\omega,\On)^M$ is the partial order whose conditions are finite partial functions $\pmap{p}{\omega}{\On^M}{par}$ ordered by reverse inclusion. 

 \item Define $\Col_*(\omega,\On)^M$ to be the (dense) suborder of $\Col(\omega,\On)^M$ consisting of all conditions $p$ with $\dom{p}\in\omega$. 

 \item Let $\Col_\geq(\omega,\On)^M$ be the notion of forcing whose conditions are finite partial functions $\pmap{p}{\omega}{\On^M\cup\{\geq\alpha\mid\alpha\in\On^M\}}{par}$,
 where $\geq\alpha$ is an element of $M$ which is not in $\On^M$ for every $\alpha\in\On^M$, and whose ordering is given by 
 $p\leq q$ if and only if $\dom{p}\supseteq\dom{q}$ and for every $n\in\dom{q}$, either
 \begin{itemize}
  \item $p(n)=q(n)$ or
  \item $q(n)$ is $\geq\alpha$ for some $\alpha\in\On^M$ and there is $\beta\geq\alpha$ such that $p(n)\in\{\beta,\geq\beta\}$.
 \end{itemize}
\end{enumerate}
\end{definition}

Note that all of these partial orders are definable over the corresponding model $M$.

\begin{notation}
Let $\PP$ be a partial order, let $\sigma$ be a $\PP$-name and let $p$ be a condition in $\PP$. Then we define the \emph{$p$-evaluation of $\sigma$} to be $$\sigma^p ~ = ~ \Set{\tau^p}{\exists q\in\PP ~ [\langle\tau,q\rangle\in\sigma ~ \wedge ~ p\leq_\PP q]}.$$
\end{notation}

The next lemma gives some basic properties of the different collapse forcings defined above.

\begin{lemma}\label{lemma:properties col}
 Let $M$ be a countable transitive model of $\ZF^-$. 
 \begin{enumerate}
  \item If $G$ is a $\Col(\omega,\On)^M$-generic filter over $M$, then for every ordinal in $M$ there is a surjection from a subset of $\omega$ onto that ordinal in $M[G]$.
 
  \item If $G$ is a $\Col_*(\omega,\On)^M$-generic filter over $M$, then $M=M[G]$.  

 \item No non-trivial maximal antichain of $\Col(\omega,\On)^M$ or $\Col_*(\omega,\On)^M$ is an element of $M$. 

 \item If $M$ is a model of $\ZFC$, then no non-trivial complete suborder of $\Col(\omega,\On)^M$ or of $\Col_*(\omega,\On)^M$ is an element of $M$.
 \item $\Col_\geq(\omega,\On)^M$ is the union of $\On^M$-many set-sized complete subforcings. 
 \end{enumerate}
\end{lemma}


\begin{proof}
 (\emph{1}) Pick  $\lambda\in\On^M$. Given $\alpha\in\On^M$, define $$D_\alpha ~ = ~ \Set{p\in\Col(\omega,\On)^M}{\exists n\in\dom{p}\,[p(n)=\alpha]}.$$ Then each $D_\alpha$ is dense in $\Col(\omega,\On)^M$ and definable over $M$. This implies that if $G$ is $\Col(\omega,\On)$-generic over $M$, then for every $\alpha\in M\cap\On$ there is an $n<\omega$ with $\{\langle n,\alpha\rangle\}\in G$. This shows that $$\sigma ~ = ~ \Set{\langle \mathsf{op}(\check{n},\check{\alpha}),\{\langle n,\alpha\rangle\}\rangle}{\alpha<\lambda, ~ n<\omega}$$ is a name for a surjection from a subset of $\omega$ onto $\lambda$.

 (2) Let $\sigma$ be a $\Col_*(\omega,\On)^M$-name in $M$. Then $\ran{p}\subseteq\rank(\sigma)$ holds for every condition $p$ in $\tc(\sigma)\cap\Col_*(\omega,\On)^M$. If we define  $$D ~ = ~ \Set{p\in \Col_*(\omega,\On)^M}{\rank(\sigma)\in\ran{p}},$$ then $D$ is dense in $\Col_*(\omega,\On)^M$ and definable over $M$. Moreover, by the above observation, we have $\sigma^G=\sigma^p\in M$, whenever $G$ is an $M$-generic filter on $\Col_*(\omega,\On)^M$ and $p\in D\cap G$, because such $p$ either extends or is incompatible to any condition in $\tc(\sigma)$.

 (3) Let $\Col$ denote either $\Col(\omega,\On)^M$ or $\Col_*(\omega,\On)^M$. Assume that $A\in M$ is an antichain of $\Col$ 
 which is not equal to $\{\one\}$. Pick $a\in A$. Now for any $b\in A\setminus\{a\}$, the domains of $a$ and $b$ cannot be 
 disjoint by incompatibility. Define $c\in\Col$ with $\dom c=\dom a$ and for every $n\in\dom c$, 
 let $c(n)=\sup\{b(n)\mid b\in A\}+1$. Hence $c$ is incompatible with every element of $A$, showing that $A$ is not maximal. 

 (4) This statement follows from the above results because our assumptions imply that set-sized partial orders in $M$ contain non-trivial maximal antichains.  

 (5) Let for every $\alpha\in\On^M$, $\Col_\geq(\omega,\alpha)$ denote the subforcing of $\Col_\geq(\omega,\On^M)$ consisting
 of finite partial functions $\pmap{p}{\omega}{\alpha\cup\{\geq\beta\mid\beta\leq\alpha\}}{par}$ with the induced ordering.
 Clearly, $$\Col_\geq(\omega,\On)^M=\bigcup_{\alpha\in\On^M}\Col_\geq(\omega,\alpha).$$
 It remains to check that for every $\alpha\in\On^M$, $\Col_\geq(\omega,\alpha)$ is a complete subforcing of $\Col_\geq(\omega,\On)^M$. 
 Let $A$ be a maximal antichain of $\Col_\geq(\omega,\alpha)$ and let $p\in\Col_\geq(\omega,\On)^M$. Consider the condition 
 $\bar p\in\Col_\geq(\omega,\alpha)$ which is obtained from $p$ by replacing $p(n)$ by $\geq\alpha$ whenever $p(n)\geq\alpha$ or $p(n)$ is of the form $\geq\beta$ for some $\beta>\alpha$. 
 Since $A$ is a maximal antichain, there is $a\in A$ such that $a$ and $\bar p$ are compatible. Let $\bar q\in\Col_\geq(\omega,\alpha)$
 be a common strengthening of $\bar p$ and $a$. But then the condition $q$ obtained from $\bar q$ by replacing $\bar q(n)$ by $p(n)$ 
 for every $n\in\dom{p}$ such that $\bar q(n)$ is of the form $\geq\alpha$ witnesses that $p$ and $a$ are compatible.  
 \end{proof}

The above computations show that, contrasting the situation with set-sized partial orders, forcing with a dense suborder of a notion of class forcing $\PP$ can produce different generic extensions than forcing with $\PP$ does.\footnote{Note that it is still true that generic filters for $\PP$ induce generic filters for its dense suborders and vice versa.}
In a subsequent paper (\cite{class_forcing3}), we will in fact show that for any notion of class forcing $\PP$, the property that all forcing notions which contain $\PP$ as a dense subforcing produce the same generic extensions as $\PP$, is essentially equivalent to the pretameness of $\PP$. 

\begin{corollary}
 If $M$ is a countable transitive model of $\ZF^-$, then there are partial orders $\PP$ and $\QQ$ definable over 
$M$ such that $\QQ$ is a dense suborder of $\PP$ and $M=M[G\cap\QQ]\subsetneq M[G]$ whenever $G$ is a $\PP$-generic filter over $M$. \qed 
\end{corollary}

In Section \ref{section:Approach} we will show that all of the partial orders that we have mentioned so far satisfy the forcing theorem.

The following notion of class forcing due to Sy Friedman is mentioned in \cite[Remark 1.8]{MR1976584}. It will be crucial for the proofs of the negative results listed in Section \ref{section:Intro}.

\begin{definition}\label{def:friedman forcing}
Let $M$ be a countable transitive model of $\ZF^-$. Define $\FF^M$ to be the partial order whose conditions are triples $p=\langle d_p, e_p,f_p\rangle$ satisfying 
\begin{enumerate}
 \item $d_p$ is a finite subset of $\omega$, 

 \item $e_p$ is a binary acyclic relation on $d_p$, 

 \item $f_p$ is an injective function with $\dom{f_p}\in\{\emptyset,d_p\}$ and $\ran{f_p}\subseteq M$, 
 \item if $\dom{f_p}=d_p$ and $i,j\in d_p$, then we have $i ~ e_p ~ j$ if and only if $f_p(i)\in f_p(j)$, 
\end{enumerate}
and whose ordering is given by $$p\leq_{\FF^M}q ~ \Longleftrightarrow ~ d_q\subseteq d_p ~ \wedge ~ e_p\cap(d_q\times d_q)=e_q ~ \wedge ~ f_q\subseteq f_p.$$ 
\end{definition}

Note that $\FF^M$ is definable over $M$.

\begin{lemma}\label{lemma:density dom}
 The set of all conditions $p$ in $\FF^M$ with $\dom{f_p}=d_p$ is dense.  
\end{lemma}

\begin{proof}
Pick $p\in\FF^M$ with $\dom{f_p}=\emptyset\neq d_p$. We inductively define (using that $e_p$ is acyclic) a function $f$ as follows. 
 For every $j\in d_p$ let
$$f(j)=\{f(i)\mid i \;e_p\; j\}\cup\{\{\emptyset,j\}\}.$$
Using that $\emptyset\not\in\range(f)$, it is easy to inductively verify that $\bar p=\langle d_p,e_p,f\rangle$ satisfies conditions (3) and (4) above, and hence is an extension of $p$ in $\FF^M$ with $\dom{f_{\bar p}}=d_{\bar{p}}$. 
\end{proof}

%


\begin{lemma}\label{lemma:friedman forcing}
 If $M$ is a countable transitive model of $\ZF^-$ and $G$ is an $\FF^M$-generic filter over $M$, then there is a binary relation $E$ on $\omega$ such that $E\in M[G]$ and the models $\langle\omega,E\rangle$ and $\langle M,\in\rangle$ are isomorphic in $\V$.  
\end{lemma}

\begin{proof} 
  Define an $\FF^M$-name $\dot E\in M$ by setting $$\dot E\,=\,\Set{\langle \mathsf{op}(\check{i},\check{j}),p_{i,j}\rangle}{ i,j\in\omega, ~ i\neq j},$$
  where $p_{i,j}$ denotes the condition in $\FF^M$ with $d_{p_{i,j}}=\{i,j\}$, $e_{p_{i,j}}=\{\langle i,j\rangle\}$ and $f_{p_{i,j}}=\emptyset$. 
  Let $G$ be an $\FF^M$-generic filter over $M$ and put $E=\dot E^G\in M[G]$. Note that $E=\bigcup\Set{e_p}{p\in G}$. Define $F=\bigcup\Set{f_p}{p\in G}$. By Lemma \ref{lemma:density dom}, the sets $D_n=\Set{p\in\FF^M}{n\in\dom{f_p}}$ are dense in $\FF^M$. Since these sets are definable over $M$, we can conclude that $F$ is injective and that $\dom{F}=\omega$.
 In order to see that $F$ is surjective, we claim that for every $x\in M$, the set $\{p\in\FF^M\mid x\in\ran{f_p}\}$
 is dense. In order to show this, let $p=\langle d_p,e_p,f_p\rangle\in\FF^M$ such that $x\notin\ran{f_p}$. Using Lemma \ref{lemma:density dom}, we may assume
 that $\dom{f_p}=d_p$. Choose $j\in\omega\setminus d_p$ and define $d_q=d_p\cup\{j\}, 
 e_q=e_p\cup\{\langle i,j\rangle\mid f_p(i)\in x\}\cup\{\langle j,i\rangle\mid x\in f_p(i)\}$ and $f_q=f_p\cup\{\langle j,x\rangle\}$.
 Then $q=\langle d_q,e_q,f_q\rangle$ is an extension of $p$ with $x\in\ran{f_q}$.
 
It remains to check that $F$ is an isomorphism between the models $\langle\omega,E\rangle$ and $\langle M,\in\rangle$. 
Take $i,j<\omega$ such that $i\,E\,j$, i.e. $p_{i,j}\in G$. By the above computations, there is a condition $p\in G$ with $i,j\in\dom{f_p}$. 
We then have $i\,e_p\,j$ and by (4) in Definition \ref{def:friedman forcing} we have $F(i)=f_p(i)\in f_p(j)=F(j)$. 
For the converse, suppose that $x,y\in M$ such that $x\in y$. By the above computations, there is a condition $p\in G$ and $i,j\in d_p$ with $F(i)=f_p(i)=x\in y=f_p(j)=F(j)$. By (4) in Definition \ref{def:friedman forcing}, this implies $i\,e_p\,j$ and therefore $i\,E\,j$ holds.  
\end{proof}



\section{The general setting}\label{section:general setting}

In the following, we outline the general setting of this paper. That is, we will actually make use 
of an approach that is slightly more general than the one presented in Section \ref{section:Intro}, 
namely one that works with models that might contain more second-order objects than just the definable ones, 
and moreover we will work with preorders instead of partial orders. 


\begin{notation}
 \begin{enumerate}
  \item We denote by $\GB^-$ the theory in the two-sorted language with variables for sets and classes, 
with the set axioms given by $\ZF^-$ with class parameters allowed in the schemata of Separation and Collection, and the class axioms of extensionality, foundation and first-order class comprehension (i.e.\ involving only set quantifiers).
Furthermore, we denote the theory $\GB^-$ enhanced with the power set axiom by $\GB$ (this is the common collection of axioms of \emph{G\"odel-Bernays set theory}).

  \item We let $\KM$ denote the axiom system of \emph{Kelley-Morse class theory}. That is, in addition to the usual $\ZFC$ axioms for sets with class parameters allowed in the schemata of Separation and Collection, one also has the class axioms of Foundation, Extensionality, Replacement, (second order) Comprehension and Global Choice. In particular, class recursion holds in models of $\KM$. For a detailed axiomatization of $\KM$, see \cite{antos2015class}.

  \item By a countable transitive model of $\GB^-$ (or $\GB$, $\KM$), we mean a model $\MM=\langle M,\C\rangle$ of $\GB^-$ (resp.\ $\GB$, $\KM$)
such that $M$ is transitive and both $M$ and $\C$ are countable in $\V$. 
 \end{enumerate}
\end{notation}


 \begin{example}
  \begin{enumerate}
   \item Let $M$ be a countable transitive model of $\ZF^-$ 
     and let $\mathrm{Def}(M)$ be the set of all subsets of $M$ that are definable over $\langle M,\in\rangle$. Then $\langle M,\mathrm{Def}(M)\rangle$ is a model of $\GB^-$. 
    \item Let $M$ be a countable transitive model of $\ZF$ satisfying $\V=\mathsf L[A]$ and Replacement for formulae
mentioning the predicate $A$, and let $\C$ be $\mathrm{Def}(M,A)$. Then $\langle M,\C\rangle$ is a model of $\GB^-$. This is the approach used in \cite{MR1780138}.
   \item Every countable transitive model of $\KM$ is a model of $\GB$.
  \end{enumerate}
 \end{example}
 
Fix a countable transitive model $\MM=\langle M,\C\rangle$ of $\GB^-$. By a \emph{notion of class forcing} (for $\MM$) we mean a preorder $\PP=\langle P,\leq_\PP\rangle$ such that $P,\leq_\PP\,\in\C$. 
We will frequently identify $\PP$ with its domain $P$. In the following, we also fix a notion of class forcing $\PP=\langle P,\leq_\PP\rangle$ for $\MM$.  

We call $\sigma$ a \emph{$\PP$-name} if all elements of $\sigma$ are of the form $\langle\tau,p\rangle$, where 
$\tau$ is a $\PP$-name and $p\in\PP$. 
Define $M^\PP$ to be the set of all $\PP$-names that are elements of $ M$ and define $\C^\PP$ to be the set of all $\PP$-names that are elements of $\C$.
In the following, we will usually call the elements of $M^\PP$ \emph{$\PP$-names} and we will call the elements of $\C^\PP$ \emph{class $\PP$-names}.
If $\sigma\in M^\PP$ is a $\PP$-name, we define 
$$\rank\sigma=\sup\{\rank\tau+1\mid\exists p\in\PP\,[\langle\tau,p\rangle\in\sigma]\}$$ 
to be its \emph{name rank}. 

We say that a filter $G$ on $\PP$ is \emph{$\PP$-generic over $\MM$} if $G$ meets every dense subset of $\PP$ that is an element of $\C$. 
Given such a filter $G$ and a $\PP$-name $\sigma$, we define the \emph{$G$-evaluation} of $\sigma$ as
$$\sigma^G=\{\tau^G\mid\exists p\in G\,[\langle\tau,p\rangle\in\sigma]\},$$
and similarly we define $\Gamma^G$ for $\Gamma\in\C^\PP$. Moreover, if $G$ is $\PP$-generic over $\MM$, then we set
$ M[G]=\Set{\sigma^G}{\sigma\in M^\PP}$ and $\C[G]=\Set{\Gamma^G}{\Gamma\in\C^\PP}$, and call $\MM[G]=\langle M[G],\C[G]\rangle$
a \emph{$\PP$-generic extension} of $\MM$.\footnote{While it does not really play any role for the present paper which second order objects we allow for in our generic extensions, we will argue in a subsequent paper (\cite{class_forcing2}) that the above choice (namely $\C[G]$) is canonical.}


For all $n<\omega$, we let $\L^n$ denote the first-order language that extends the language of set theory $\L_\in$ by unary predicate symbols $A_0,\ldots,A_{n-1}$. Given an $\L^n$-formula $\varphi(v_0,\ldots,v_{m-1})$, a tuple $\vec{\Gamma}=\langle\Gamma_0,\ldots,\Gamma_{n-1}\rangle\in(\C^\PP)^n$, a condition $p\in P$ and names $\sigma_0,\ldots,\sigma_{m-1}\in M^\PP$, we write 
\begin{equation*}
 p\Vdash^{\MM,\vec{\Gamma}}_\PP\varphi(\sigma_0,\ldots,\sigma_{m-1}) 
\end{equation*}
to denote that $\varphi(\sigma_0^G,\ldots,\sigma_{m-1}^G)$ holds in the structure $ M_{\vec{\Gamma}}[G]=\langle M[G], \in, \Gamma_0^G,\ldots,\Gamma_{n-1}^G\rangle$ 
whenever $G$ is a $\PP$-generic filter over $\MM$ with $p\in G$. 
Whenever the context is clear, we will omit the superscripts and subscripts.  

Our choice of considering preorders instead of partial orders is due to the reason that in the case of a two-step iteration $\PP*\dot\QQ$ of notions of class forcing, as defined in \cite{MR1780138} (see also Section \ref{sec:ftl} of the present paper), we will have conditions of the form $\langle p,\dot{q}\rangle$ for $p\in\PP$ and $p$ forcing that $\dot{q}\in\dot\QQ$. In general there will be distinct pairs $\langle p,\dot{q_0}\rangle$ and $\langle p,\dot{q_1}\rangle$ such that $p\Vdash_\PP\dot{q_0}=\dot{q_1}$, i.e.\ one naturally obtains a preorder that is not antisymmetric.
However, in some contexts it will become crucial for our orderings to be antisymmetric. 
In that case we will use the following additional property:

\begin{definition}
We say that a model $\langle M,\C\rangle$ of $\GB^-$ satisfies \emph{representatives choice}, if for every equivalence relation 
$E\in\C$ there is $A\in\C$ and a surjective map $\pi:\dom E\ra A$ in $\C$ such that 
$\langle x,y\rangle\in E$ if and only if $\pi(x)=\pi(y)$. 
\end{definition}

Using representatives choice, given a preorder $\PP=\langle P,\leq_\PP\rangle\in\C$, by considering the equivalence relation
$$p\approx q\quad\text{iff}\quad p\leq_\PP q\wedge q\leq_\PP p,$$ we obtain a partial order $\QQ\in C$ and 
a surjective map $\pi:\PP\ra\QQ$ in $\C$ such that for all $p,q\in\PP$, $p\approx q$ if and only if $\pi(p)=\pi(q)$. 

Clearly, representatives choice follows from the existence of a global well-order. Furthermore, if $M$ satisfies the power
set axiom, then we also obtain representatives choice, since we can use Scott's trick to obtain the 
sets $[p]=\{q\in\PP\mid q\approx p\wedge\forall r\,[q\approx r\ra\rank(q)\leq\rank(r)]\}\in M$ for $p\in\PP$.


\section{The forcing theorem}

In this section, we fix a countable transitive model $\MM=\langle M,\C\rangle$ of $\GB^-$ and a notion of class forcing $\PP=\langle P,\leq_\PP\rangle$ for $\MM$.
We will show that in order to obtain the forcing theorem for all $\L^n$-formulae, it suffices that the forcing relation 
for either the formula $\anf{v_0\in v_1}$ or the formula $\anf{v_0=v_1}$ is definable, thus proving Theorem \ref{theorem:ForcingTheorem}. 

\begin{definition}\label{def:ft}
 Let $\varphi\equiv\varphi(v_0,\ldots,v_{m-1})$ be an $\L^n$-formula. 

 \begin{enumerate}


  \item We say that \emph{$\PP$ satisfies the definability lemma for $\varphi$ over $\MM$} if 
   \begin{equation*}
  \Set{\langle p,\sigma_0,\ldots,\sigma_{m-1}\rangle\in P\times M^\PP\times\ldots\times M^\PP}{p\Vdash^{\MM,\vec{\Gamma}}_\PP\varphi(\sigma_0,\ldots,\sigma_{m-1})}\in\C 
 \end{equation*} 
 for all $\vec{\Gamma}\in(\C^\PP)^n$. 
 
 \item We say that \emph{$\PP$ satisfies the truth lemma for $\varphi$ over $\MM$} if for all $\sigma_0,\ldots,\sigma_{m-1}\in M^\PP$, $\vec{\Gamma}\in(\C^\PP)^n$ and every filter $G$ which is $\PP$-generic over $\MM$ with 
  \begin{equation*}
    M_{\vec{\Gamma}}[G]\models\varphi(\sigma_0^G,\ldots,\sigma_{m-1}^G), 
  \end{equation*}
  there is $p\in G$ with $p\Vdash^{\MM,\vec{\Gamma}}_\PP\varphi(\sigma_0,\ldots,\sigma_{m-1})$.

  \item We say that \emph{$\PP$ satisfies the forcing theorem for $\varphi$ over $\MM$} if $\PP$ satisfies both the definability lemma and the truth lemma for $\varphi$ over $\MM$.
 \end{enumerate}
 \end{definition}


Our goal is to prove a generalization of Theorem \ref{theorem:ForcingTheorem} in our general setting
which allows for second-order objects. The first step to achieve this is to show that the definability lemma
for some atomic formula already implies the truth lemma to hold for all atomic formulae. 

\begin{lemma}\label{lem:def subs ft}
 Assume that $\PP$ satisfies the definability lemma for $\anf{v_0\in v_1}$ or $\anf{v_0=v_1}$ over $\MM$. 
 Then $\PP$ satisfies the forcing theorem for all atomic formulae. 
%
\end{lemma}

\begin{proof}
Suppose first that the definability lemma holds for $\anf{v_0\in v_1}$. We denote 
by $p\Vdash_\PP^{\MM,*}\sigma\subseteq\tau$ the statement that for all $\langle\rho,r\rangle\in\sigma$ and 
for all $q\leq_\PP p,r$, the set 
$$D_{\rho,\tau}=\{s\in\PP\mid s\Vdash_\PP^\MM\rho\in\tau\}$$
is dense below $q$ in $\PP$. 
Furthermore, let $p\Vdash_\PP^{\MM,*}\sigma=\tau$ denote that $p\Vdash_\PP^{\MM,*}\sigma\subseteq\tau$ and $p\Vdash_\PP^{\MM,*}\tau\subseteq\sigma$. 
We show by induction on the lexicographic order on pairs $\langle\rank(\sigma)+\rank(\tau),\rank(\sigma)\rangle$ that the following hold for each $p\in\PP$:
\begin{enumerate-(1)}
 \item $p\Vdash_\PP^\MM\sigma\in\tau$ if and only if the set 
 $$E_{\sigma,\tau}=\{q\in\PP\mid\exists\langle\rho,r\rangle\in\tau\,[q\leq_\PP r\wedge q\Vdash_\PP^{\MM,*}\sigma=\rho]\}\in\C$$
 is dense below $p$ in $\PP$.
 \item $p\Vdash_\PP^\MM\sigma\subseteq\tau$ if and only if $p\Vdash_\PP^{\MM,*}\sigma\subseteq\tau$. 
 In particular, $p\Vdash_\PP^\MM\sigma=\tau$ if and only if $p\Vdash_\PP^{\MM,*}\sigma=\tau$. 
 \item There is a dense subset of $\PP$ in $\C$ that consists of conditions $p$ in $\PP$ such that either 
 $p\Vdash_\PP^\MM\sigma\in\tau$ or $p\Vdash_\PP^\MM\sigma\notin\tau$.
 \item There is a dense subset of $\PP$ in $\C$ that consists of conditions $p$ in $\PP$ such that either 
 $p\Vdash_\PP^\MM\sigma\subseteq\tau$ or $p\Vdash_\PP^\MM\sigma\nsubseteq\tau$. 
\end{enumerate-(1)}
To start the induction, note that if $\rank(\sigma)+\rank(\tau)=0$, then (1)--(4) trivally hold.
Note that (3) implies that the truth lemma holds for $\anf{v_0\in v_1}$. Furthermore, (4) implies
the truth lemma for $\anf{v_0\subseteq v_1}$ and hence also for equality.
Suppose now that (1)--(4) are satisfied for all pairs of names $\langle\bar\sigma,\bar\tau\rangle$ in $M^\PP$ for which $\langle\rank(\bar\sigma)+\rank(\bar\tau),\rank(\bar\sigma)\rangle$ is lexicographically less than $\langle\rank(\sigma)+\rank(\tau),\rank(\sigma)\rangle$, that is $\rank(\bar\sigma)+\rank(\bar\tau)\le\rank(\sigma)+\rank(\tau)$ and in case of equality, we have that $\rank(\bar\sigma)<\rank(\sigma)$. 

In order to prove (1), pick a condition $p\in\PP$ with $p\Vdash_\PP^\MM\sigma\in\tau$ and $q\leq_\PP p$. Let $G$ be $\PP$-generic 
over $\MM$ with $q\in G$. Then $\sigma^G\in\tau^G$ by assumption and hence there is $\langle\rho,r\rangle\in\tau$ 
with $r\in G$ and $\sigma^G=\rho^G$. By our inductive assumption, property (4) yields a condition 
$s\in G$ with $s\Vdash_\PP^\MM\sigma=\rho$ which by (2) is equivalent to $s\Vdash_\PP^{\MM,*}\sigma=\rho$.
Since $G$ is a filter, there is $t\in G$ with $t\leq_\PP q,r,s$. In particular, $t\in E_{\sigma,\tau}$. 
For the other direction, suppose that $E_{\sigma,\tau}$ is dense below $p$. Let $G$ be $\PP$-generic over $\MM$
with $p\in G$. By density of $E_{\sigma,\tau}$ we can take $q\in G$ and $\langle\rho,r\rangle\in\tau$ 
such that $q\leq_\PP r$ and $q\Vdash_\PP^{\MM,*}\sigma=\rho$. Then $r\in G$ and so $\rho^G\in\tau^G$. 
Thus by our inductive assumption, condition (2) implies that $\sigma^G=\rho^G\in\tau^G$.

For (2), suppose first that $p\Vdash_\PP^\MM\sigma\subseteq\tau$, let $\langle\rho,r\rangle\in\sigma$ and let $q\leq_\PP p,r$.
Take a $\PP$-generic filter $G$ with $q\in G$.
Then $\rho^G\in\sigma^G\subseteq\tau^G$. By our inductive assumption,
we can find $s\in G$ so that $s\Vdash_\PP^\MM\rho\in\tau$. Given any $q^*\leq_\PP q$, by strengthening $s$
if necessary, we can find such $s\leq_\PP q^*$, as desired.  
Conversely, assume that $p\Vdash_\PP^{\MM,*}\sigma\subseteq\tau$ and let $G$ be $\PP$-generic over $\MM$
with $p\in G$. Let $\langle\rho,r\rangle\in\sigma$ with $r\in G$. We have to show that $\rho^G\in\tau^G$. 
Let $q\in G$ be a common strengthening of $p$ and $r$. Then by assumption, the set 
$D_{\rho,\tau}$ is dense below $q$. By genericity, we can take $s\in D_{\rho,\tau}\cap G$.
Using our inductive assumption, this shows that $\rho^G\in\tau^G$, as desired. 

For (3), consider the set
$$D=\{p\in\PP\mid\forall\langle\rho,r\rangle\in\tau\,\forall q\leq_\PP p,r\ [q\nVdash_\PP^{\MM}\sigma=\rho]\}.$$
Then our inductive assumptions imply that $D\in\C$. Moreover, condition (1) states that $D$ is nonempty below every
$p\in\PP$ with $p\nVdash_\PP^\MM\sigma\in\tau$. 
Hence it suffices to show that $p\Vdash_\PP^\MM\sigma\notin\tau$ for every $p\in D$, since then 
$D\cup\{p\in\PP\mid p\Vdash_\PP^\MM\sigma\in\tau\}\in\C$ is a dense set of conditions deciding $\sigma\in\tau$.
So take $p\in D$ and suppose that $p\nVdash_\PP^\MM\sigma\notin\tau$. Then there is a $\PP$-generic filter $G$
containing $p$ such that $\sigma^G\in\tau^G$. Then there must be $\langle\rho,r\rangle\in\tau$ with $r\in G$
and $\sigma^G=\rho^G$. By our inductive assumption, we can find $q\in G$ with $q\Vdash_\PP^\MM\sigma=\rho$.
By possibly strengthening $q$ using that $G$ is a filter, we may assume that $q\leq_\PP p,r$. But this contradicts that $p\in D$. 

In order to verify (4), we define 
$$E=\{p\in\PP\mid\exists\langle\rho,r\rangle\in\sigma\ [p\leq_\PP r\wedge\forall q\leq_\PP p\ (q\nVdash_\PP^\MM\rho\in\tau)]\}.$$
As above, $E$ is in $\C$ inductively, and it is nonempty below every condition which does not force $\sigma\subseteq\tau$. 
As in the proof of (3) it remains to check that $p\Vdash_\PP^\MM\sigma\nsubseteq\tau$ for each $p\in E$. 
Assume, towards a contradiction, that there is $p\in E$ with $p\nVdash_\PP^\MM\sigma\nsubseteq\tau$. Then 
there is a $\PP$-generic filter with $p\in G$ and $\sigma^G\subseteq\tau^G$. Let $\langle\rho,r\rangle$ witness that 
$p\in E$. Then $r\in G$ and so $\rho^G\in\sigma^G\subseteq\tau^G$. Using (3) inductively, we obtain $q\in G$ with $q\Vdash_\PP^\MM\rho\in\tau$.
But then there is $s\leq_\PP p,q$, contradicting that $p\in E$. 

If the definability lemma holds for $\anf{v_0=v_1}$, we can define the $\PP$-forcing relation for 
$\anf{v_0\in v_1}$ by stipulating (as above) that $p\Vdash_\PP^\MM\sigma\in\tau$ if and only if the set 
\[\{q\in\PP\mid\exists\langle\rho,r\rangle\in\tau\,[q\leq_\PP r\wedge q\Vdash_\PP^\MM\sigma=\rho]\}\] is dense below $p$. 
\end{proof}

\begin{theorem}\label{thm:ft}
If $\PP$ satisfies the definability lemma either for $\anf{v_0\in v_1}$ or for $\anf{v_0=v_1}$ over $\MM$, then $\PP$ satisfies the forcing theorem for every $\L^n$-formula over $\MM$. 
\end{theorem}

\begin{proof}
By the previous lemma, we already know that $\PP$ satisfies the forcing theorem for all atomic formulae. Let us next consider formulas of the form $\anf{v_0\in V_1}$, involving a class variable $V_1$.

Let $\sigma\in M^\PP$ and $\Gamma\in\C^\PP$. We claim that $p\Vdash_\PP^{\MM,\Gamma}\sigma\in\Gamma$ if and only if the set
$$D=\{q\in\PP\mid\exists\langle\tau,r\rangle\in\Gamma\,[q\leq_\PP r\wedge q\Vdash_\PP^\MM\sigma=\tau]\}$$
is dense below $p$. Note that $D\in\C$ since the forcing relation for equality is definable. 
First assume that $p\Vdash_\PP^{\MM,\Gamma}\sigma\in\Gamma$ and let $q\leq_\PP p$. Let $G$ be a $\PP$-generic filter with 
$q\in G$. Then $\sigma^G\in\Gamma^G$, i.e. there is $\langle\tau,r\rangle\in\Gamma$ such that $r\in G$ and $\sigma^G=\tau^G$.
By the truth lemma for $\anf{v_0=v_1}$, there is $s\in G$ such that $s\Vdash_\PP^\MM\sigma=\tau$. But then every
$t\leq_\PP q,r,s$ is in $D$. 
Conversely, if $D$ is dense in $\PP$ and $G$ is $\PP$-generic over $\MM$ with $p\in G$, then we find
$q\leq_\PP p$ in $D\cap G$. By definition of $D$ there is $\langle\tau,r\rangle\in\Gamma$ such that 
$q\leq_\PP r$ and $q\Vdash_\PP\sigma=\tau$. Thus using that $q,r\in G$ we get $\sigma^G=\tau^G\in\Gamma^G$. 
The truth lemma for $\anf{v_0\in V_1}$ follows from the truth lemma for equality.

For composite $\L^n$-formulae, we can define the forcing relation by the usual recursion:
\begin{align*}
 \hspace{0.6cm}&p\Vdash_\PP^{\MM,\vec{\Gamma}}(\varphi\wedge\psi)(\sigma_0,\dots,\sigma_{m-1})&\Llr\quad& p\Vdash_\PP^{\MM,\vec{\Gamma}}\varphi(\sigma_0,\dots,\sigma_{m-1})\text{ and }p\Vdash_\PP^{\MM,\vec\Gamma}\psi(\sigma_0,\dots,\sigma_{m-1})\hspace{0.3cm}\\
 &p\Vdash_\PP^{\MM,\vec{\Gamma}}\neg\varphi(\sigma_0,\dots,\sigma_{m-1})&\Llr\quad&\forall q\leq_\PP p\, [q\nVdash_\PP^{\MM,\vec{\Gamma}}\varphi(\sigma_0,\dots,\sigma_{m-1})]\\
 &p\Vdash_\PP^{\MM,\vec{\Gamma}}\forall x\varphi(\sigma_0,\dots,\sigma_{m-1},x)&\Llr\quad&\forall\tau\in M^\PP\, [p\Vdash_\PP^{\MM,\vec{\Gamma}}\varphi(\sigma_0,\dots,\sigma_{m-1},\tau)],
\end{align*}
where $\sigma_0,\dots,\sigma_{m-1}\in M^\PP$ and $\vec{\Gamma}\in(\C^\PP)^n$. The truth lemma can be verified as for set forcing in each case.
\end{proof}



\section{Boolean completions}

In set forcing, every partial order has a unique Boolean completion whose elements are the regular open subsets
of the partial order. In this section we will investigate the relationship between the existence of a Boolean completion and the forcing theorem for notions of class forcing. 
Let $\MM=\langle M,\C\rangle$ be a fixed countable transitive model of $\GB^-$ and let $\PP=\langle P,\leq_\PP\rangle$ be a notion of class forcing.

Let $\L_{\On,0}$ denote the infinitary quantifier-free language that allows for set-sized
conjunctions and disjunctions. By $\L_{\On,0}^{\Vdash}(\PP,M)$ we denote the language of infinitary quantifier-free
formulae in the forcing language of $\PP$ over $M$, that allows reference to the generic predicate $G$. More precisely,
its constants are all elements of $M^\PP$, and it has an additional predicate $\dot G$. 
We define $\L_{\On,0}^{\Vdash}(\PP,M)$ and the class $\Fm_{\On,0}^{\Vdash}(\PP,M)$ of G\"odel codes of $\L_{\On,0}^{\Vdash}(\PP,M)$-formulae by simultaneous recursion:
\begin{enumerate}
 \item Atomic $\L_{\On,0}^{\Vdash}(\PP,M)$-formulae are of the form $\sigma=\tau,\sigma\in\tau$ or
 $\check p\in\dot G$ for $\sigma,\tau\in M^\PP$
and $p\in\PP$, where $\dot G=\{\langle\check p,p\rangle\mid p\in\PP\}\in\C^\PP$ is the canonical
class name for the generic filter. G\"odel codes of atomic $\L_{\On,0}^{\Vdash}(\PP,M)$-formulae are given by
\begin{align*}
 \gbr{\check p\in \dot G}&=\langle0,p\rangle\\
 \gbr{\sigma=\tau}&=\langle1,\sigma,\tau\rangle\\
 \gbr{\sigma\in\tau}&=\langle2,\sigma,\tau\rangle.
\end{align*}
\item If $\varphi$ is an $\L_{\On,0}^{\Vdash}(\PP,M)$-formula, then so is $\neg\varphi$, and its G\"odel code is given by 
\begin{align*}
 \gbr{\neg\varphi}=\langle3,\gbr{\varphi}\rangle.
\end{align*}
\item If $I\in M$ and for every $i\in I$, $\varphi_i$ is an $\L_{\On,0}^{\Vdash}(\PP,M)$-formula
such that $\langle\gbr{\varphi_i}\mid i\in I\rangle\in M$, 
then so are $\bigvee_{i\in I}\varphi_i$ and $\bigwedge_{i\in I}\varphi_i$ and their G\"odel codes are given by
\begin{align*}
 \gbr{\bigvee_{i\in I}\varphi_i}&=\langle4,I,\{\langle i,\gbr{\varphi_i}\rangle\mid i\in I\}\rangle\\
 \gbr{\bigwedge_{i\in I}\varphi_i}&=\langle5,I,\{\langle i,\gbr{\varphi_i}\rangle\mid i\in I\}\rangle.
\end{align*}
\end{enumerate}
Now define $\Fm_{\On,0}^{\Vdash}(\PP,M)\in\C$ to be the class of all G\"odel codes of infinitary formulae in the forcing language of $\PP$ over $M$.
If $G$ is a $\PP$-generic filter over $\MM$ and $\varphi$ is an $\L_{\On,0}^{\Vdash}(\PP,M)$-formula,
then we write $\varphi^G$ for the formula obtained from $\varphi$ by replacing each 
$\PP$-name $\sigma$ occurring in $\varphi$ by its evaluation $\sigma^G$, and by evaluating $\dot G$ as $G$. Note that 
$\varphi^G$ is a formula in the infinitary language $\L_{\On,0}$ with an additional predicate for the generic $G$.  
Given an $\L_{\On,0}^{\Vdash}(\PP,M)$-formula $\varphi$ and $p\in\PP$, we write 
$p\Vdash_\PP^{\MM}\varphi$
to denote that $\langle M[G],\in,G\rangle\models\varphi^G$ whenever $G$ is a $\PP$-generic
filter over $\MM$ with $p\in G$. 

\begin{definition}\label{def:uft}
We say $\PP$ \emph{satisfies the uniform forcing theorem for $\L_{\On,0}^{\Vdash}(\PP,M)$-formulae} if 
$$\{\langle p,\gbr{\varphi}\rangle\in P\times\Fm_{\On,0}^{\Vdash}(\PP,M)\mid p\Vdash_\PP^\MM\varphi\}\in\C$$
and $\PP$ satisfies the truth lemma for every $\L_{\On,0}^{\Vdash}(\PP,M)$-formula $\varphi$ over $\MM$, i.e.\ 
for every $\PP$-generic filter $G$ over $\MM$, if $M[G]\models\varphi^G$ then there is $p\in G$ such that $p\Vdash_\PP^\MM\varphi$.
\end{definition}

The following lemma will allow us to infer that the uniform forcing theorem for infinitary formulae is equivalent to the forcing 
theorem for equality. 

\begin{lemma}\label{lem:names uft}
 There is an assignment $$\Fm_{\On,0}^{\Vdash}(\PP,M)\ra M^\PP\times M^\PP; ~ \gbr{\varphi}\mapsto\langle\nu_\gbr{\varphi},\mu_\gbr{\varphi}\rangle$$ 
such that $\{\langle\gbr{\varphi},\nu_\gbr{\varphi},\mu_\gbr{\varphi}\rangle\mid\gbr{\varphi}\in\Fm_{\On,0}^{\Vdash}(\PP,M)\}\in\C$ and 
\begin{equation}\label{al:mu nu}
 \one_\PP\Vdash_\PP(\varphi\lr\nu_\gbr{\varphi}=\mu_\gbr{\varphi})
\end{equation}
for every $\varphi\in\mathcal L_{\On,0}^{\Vdash}(\PP,M)$. 
\end{lemma}

\begin{proof}
We will argue by induction that, given names $\nu_\psi$ and $\mu_\psi$ satisfying \eqref{al:mu nu} for every proper subformula $\psi$ of $\varphi$, we can, uniformly in $\gbr{\varphi}$, define $\nu_\gbr{\varphi}$ and $\mu_\gbr{\varphi}$ such that \eqref{al:mu nu} holds.

Observe that since $\neg\bigwedge_{i\in I}\varphi_i\equiv\bigvee_{i\in I}\neg\varphi_i$ and 
$\neg\bigvee_{i\in I}\varphi_i\equiv\bigwedge_{i\in I}\neg\varphi_i$ we can assume that all formulae are in negation normal form,
i.e.\ the negation operator is applied to atomic formulae only. Next, due to the equivalences 
\begin{align*}
 \sigma\neq\tau&\equiv\sigma\nsubseteq\tau\vee\tau\nsubseteq\sigma,\\
 \sigma\nsubseteq\tau&\equiv\bigvee_{\langle\pi,p\rangle\in\sigma}(\pi\notin\tau\wedge\check{p}\in\dot{G}),\\
 \sigma\notin\tau&\equiv\bigwedge_{\langle\pi,p\rangle\in\tau}(\sigma\neq\pi\vee\check{p}\notin\dot{G}),
\end{align*}
we can further suppose that the only negated formulae are of the form $\check{p}\notin\dot G$. 

For the atomic cases, let
\begin{align*}
 \hspace{5cm}\nu_\gbr{\check{p}\in\dot{G}}&=\{\langle\check{0},p\rangle\},&\mu_\gbr{\check{p}\in\dot{G}}&=\check{1},\hspace{5cm}\\
 \nu_\gbr{\check{p}\notin\dot{G}}&=\emptyset,&\mu_\gbr{\check{p}\notin\dot{G}}&=\{\langle\check{0},p\rangle\},\\
 \nu_\gbr{\sigma=\tau}&=\sigma,&\mu_\gbr{\sigma=\tau}&=\tau,\\
 \nu_\gbr{\sigma\in\tau}&=\tau,&\mu_\gbr{\sigma\in\tau}&=\tau\cup\{\langle\sigma,\one_\PP\rangle\}.
\end{align*}
It is easy to check that \eqref{al:mu nu} holds for all atomic formulae.

If $\varphi$ is a conjunction of the form $\bigwedge_{i\in I}\varphi_i$ and $\nu_\gbr{\varphi_i},\mu_\gbr{\varphi_i}$ have already been defined for $i\in I$, let
\begin{align*}
 \nu_\gbr{\varphi}&=\{\langle\op(\nu_\gbr{\varphi_i},\check{i}),\one_\PP\rangle\mid i\in I\}\textrm{ and}\\
 \mu_\gbr{\varphi}&=\{\langle\op(\mu_\gbr{\varphi_i},\check{i}),\one_\PP\rangle\mid i\in I\}.
\end{align*}
If $G$ is $\PP$-generic over $M$ and $M[G]\models\varphi^G$, then $M[G]\models\varphi_i^G$ for all $i\in I$. 
By assumption, this means that $\nu_\gbr{\varphi_i}^G=\mu_\gbr{\varphi_i}^G$ for every $i\in I$, thus also 
$\nu_\gbr{\varphi}^G=\mu_\gbr{\varphi}^G$. The converse is similar. 

Next suppose that $\varphi$ is of the form $\bigvee_{i\in I}\varphi_i$. 
Let $\bar\nu_\gbr{\varphi_i}=\op(\nu_\gbr{\varphi_i},\check i)$ and $\bar\mu_\gbr{\varphi_i}=\op(\mu_\gbr{\varphi_i},\check i)$
for each $i\in I$.
Let
\begin{align*}
 \pi_\gbr{\varphi}&=\{\langle\op(\bar\nu_\gbr{\varphi_i},\bar\mu_\gbr{\varphi_i}),\one_\PP\rangle\mid i\in I\}\cup\{\langle\op(\bar\nu_\gbr{\varphi_i},\bar\nu_\gbr{\varphi_i}),\one_\PP\rangle\mid i\in I\},\\
 \nu_\gbr{\varphi}^i&=\pi_\gbr{\varphi}\setminus\{\langle\op(\bar\nu_\gbr{\varphi_i},\bar\mu_\gbr{\varphi_i}),\one_\PP\}.
\end{align*}
Now we define 
\begin{align*}
 \nu_\gbr{\varphi}&=\{\langle\nu_\gbr{\varphi}^i,\one_\PP\rangle\mid i\in I\}\\
 \mu_\gbr{\varphi}&=\nu_\gbr{\varphi}\cup\{\langle\pi_\gbr{\varphi},\one_\PP\rangle\}.
\end{align*}
If $G$ is $\PP$-generic and $M[G]\models\varphi^G$ there is 
some $i\in I$ such that $M[G]\models\varphi_i^G$. By induction, this implies that 
$M[G]\models\nu_\gbr{\varphi_{i}}^G=\mu_\gbr{\varphi_{i}}^G$. Thus 
$\pi_\gbr{\varphi}^G=(\nu_\gbr{\varphi}^i)^G$ and $\nu_\gbr{\varphi}^G=\mu_\gbr{\varphi}^G$. 
For the converse, suppose that there is a generic $G$ such that $M[G]\models\neg\varphi^G$, 
hence for every $i\in I$, $M[G]\models\neg\varphi_i^G$. But then in $M[G]$, 
for every $i\in I$, we have $\nu_\gbr{\varphi_i}^G\neq\mu_\gbr{\varphi_i}^G$.
Therefore $\pi_\gbr{\varphi}^G$ is not of the form $(\nu_\gbr{\varphi}^i)^G$ for any $i\in I$,
which shows that $\pi_\gbr{\varphi}^G\in\mu_\gbr{\varphi}^G\setminus\nu_\gbr{\varphi}^G$.
\end{proof}

Next, we will use the above lemma to provide a characterization of notions of class forcing which satisfy the forcing theorem, using Boolean completions. 

\begin{definition}
If $\BB$ is a Boolean algebra, then we say that $\BB$ is \emph{$M$-complete} if the 
supremum $\sup_\BB A$ of all elements in $A$ exists in $\BB$ for every $A\in M$ with $A\subseteq\BB$. 
\end{definition}

\begin{definition} 
We say that \emph{$\PP$ has a Boolean completion in $\MM$} if there is an $M$-complete Boolean algebra 
$\BB=\langle B,0_\BB,1_\BB\,\neg,\wedge,\vee\rangle$ such that $B$, all Boolean operations of $\BB$ and an injective dense embedding from $\PP$ into $\BB\setminus\{0_\BB\}$ are elements of $\C$.
\end{definition}

\begin{theorem}\label{thm:compl forcing thm} Assume that $\MM$ satisfies representatives choice and let 
$\PP=\langle P,\leq_\PP\rangle$ be a separative and antisymmetric notion of class forcing for $\MM$. Then the following statements (over $\MM$) are equivalent:
\begin{enumerate}
 \item $\PP$ satisfies the definability lemma for one of the $\L_\in$-formulae $\anf{v_0\in v_1}$ or $\anf{v_0=v_1}$.
 \item $\PP$ satisfies the forcing theorem for all $\L_\in$-formulae.
 \item $\PP$ satisfies the uniform forcing theorem for all $\L_{\On,0}^{\Vdash}(\PP,M)$-formulae.
 \item $\PP$ has a Boolean completion. 
\end{enumerate}

Moreover separativity and antisymmetricity of $\PP$ are only necessary for the implication from (3) to (4), in particular the equivalence of (1)--(3) and those being a consequence of (4) holds as well without these assumptions. In fact, without these assumptions, (4) implies that $\PP$ is separative and antisymmetric.
\end{theorem}

\begin{proof}
That (1) implies (2) is exactly the statement of Theorem \ref{thm:ft}.
We start by showing that (2) implies (3). In fact, we only need to assume the forcing theorem for equality. By Lemma \ref{lem:names uft}, 
for every $\L_{\On,0}^{\Vdash}(\PP,M)$-formula $\varphi$ there are $\PP$-names $\mu_\gbr{\varphi}$ and $\nu_\gbr{\varphi}$
such that $\{\langle\gbr{\varphi},\mu_\gbr{\varphi},\nu_\gbr{\varphi}\rangle\mid\gbr{\varphi}\in\Fm_{\On,0}^{\Vdash}(\PP,M)\}\in\C$
and $\one_\PP\Vdash_\PP(\varphi\lr\mu_\gbr{\varphi}=\nu_\gbr{\varphi})$. Therefore, we can define the forcing relation for $\L_{\On,0}^{\Vdash}(\PP,M)$-formulae by stipulating
\begin{align*}
 p\Vdash_\PP\varphi\text ~ \Longleftrightarrow ~ p\Vdash_\PP\nu_\gbr{\varphi}=\mu_\gbr{\varphi}.
\end{align*}
This makes the truth lemma for $\varphi$ an immediate consequence of the truth lemma for equality.
As the definability lemma holds for equality and $\{\langle\gbr{\varphi},\mu_\gbr{\varphi},\nu_\gbr{\varphi}\rangle\mid\gbr{\varphi}\in\Fm_{\On,0}^{\Vdash}(\PP,M)\}\in\C$, 
we have $\{\langle p,\gbr{\varphi}\rangle\mid p\Vdash_\PP\varphi\}=\{\langle p,\gbr{\varphi}\rangle\mid p\Vdash_\PP\nu_\gbr{\varphi}=\mu_\gbr{\varphi}\}\in\C$, thus proving the uniform forcing theorem for $\L_{\On,0}^{\Vdash}(\PP,M)$-formulae. 

Assume now that (3) holds. We will construct what could be seen as an analogue of the Lindenbaum algebra. 
Define a Boolean algebra $\BB$ in the following way:
Consider the class $\Fm_{\On,0}^{\Vdash}(\PP,M)$ of all G\"odel codes of infinitary formulae in the forcing language of $\PP$ 
endowed with the canonical Boolean operations, i.e.\ 
suprema and infima are just set-sized disjunctions and conjunctions of formulae and complements are just negations. 
In order to obtain a complete Boolean algebra from $\Fm_{\On,0}^{\Vdash}(\PP,M)$, consider the equivalence relation 
$$\gbr{\varphi}\approx\gbr{\psi}\quad\text{if and only if}\quad\one_\PP\Vdash_\PP\varphi\lr\psi.$$
Since $\MM$ satisfies representatives choice, there are $\BB\in\C$ and $\pi\in\C$ such that $\pi:\Fm_{\On,0}^{\Vdash}(\PP,M)\ra\BB$ is surjective and such that $\pi(\gbr{\varphi})=\pi(\gbr{\psi})$ if and only if $\gbr{\varphi}\approx\gbr{\psi}$. 
Now we can obtain induced Boolean operations on $\BB$ in the obvious way and define 
$0_\BB=\pi(\gbr{0\neq0})$ and $\one_\BB=\pi(\gbr{0=0})$. Clearly, $\BB$ is an $M$-complete Boolean algebra. 
We identify $p\in\PP$ with the formula $\pi(\gbr{\check{p}\in\dot{G}})$, thus obtaining a dense embedding $i:\PP\ra\BB$ in $\C$. 
Note that the injectivity of $i$ follows from the antisymmetry and separativity of $\PP$.

Finally, we prove that (4) implies (1). Assume that $\PP$ has a Boolean completion $\BB(\PP)$. Without loss of generality,
we consider $\PP$ to be a subset of $\BB(\PP)$. For $\sigma,\tau\in M^\PP$, we recursively define the Boolean values
\begin{align*}
 \llbracket\sigma\in\tau\rrbracket&=\text{sup}_\PP\{\llbracket\sigma=\pi\rrbracket\wedge p\mid\langle\pi,p\rangle\in\tau\},\\
 \llbracket\sigma=\tau\rrbracket&=\llbracket\sigma\subseteq\tau\rrbracket\wedge\llbracket\tau\subseteq\sigma\rrbracket\textrm{ and}\\
 \llbracket\sigma\subseteq\tau\rrbracket&=\text{inf}_\PP\{\neg\llbracket\pi\in\sigma\rrbracket\vee\llbracket\pi\in\tau\rrbracket\mid\pi\in\dom\sigma\}.
\end{align*}
Clearly, the set $\{\langle\sigma,\tau,\llbracket\varphi(\sigma,\tau)\rrbracket\rangle\mid\sigma,\tau\in M^\PP\}$ is in $\C$ for every atomic $\L_\in$-formula $\varphi(v_0,v_1)$.
We claim that for every atomic formula $\varphi(v_0,v_1)$, $\sigma,\tau\in M^\PP$ and for every $\PP$-generic $G$,
\begin{equation}\label{al:ft}
 M[G]\models\varphi(\sigma^G,\tau^G)\quad\text{iff}\quad\exists p\in G(p\leq_{\BB(\PP)}\llbracket\varphi(\sigma,\tau)\rrbracket)
\end{equation}
This can be shown inductively by considering the recursive construction of $\llbracket\varphi(\sigma,\tau)\rrbracket$. 
For example, assume that $M[G]\models\sigma^G\in\tau^G$, i.e.\ there 
is $\langle\pi,p\rangle\in\tau$ such that $p\in G$ and $\sigma^G=\pi^G$. Thus inductively, there is 
$r\in G$ such that $r\leq_{\BB(\PP)}\llbracket\sigma=\pi\rrbracket$. Let $q\in G$ be such that $q\leq_\PP p,r$, so 
$q\leq_{\BB(\PP)}\llbracket\sigma=\pi\rrbracket\wedge p\leq_{\BB(\PP)}\llbracket\sigma\in\tau\rrbracket$. 
Conversely, assume there is $p\in G$ such that $p\leq_{\BB(\PP)}\llbracket\sigma\in\tau\rrbracket$. By the definition of $\llbracket\sigma\in\tau\rrbracket$, there is 
a strengthening $q\in G$ of $p$ and $\langle\pi,r\rangle\in\tau$ such that $q\leq_{\BB(\PP)}\llbracket\sigma=\pi\rrbracket\wedge r$.
Then by induction, $\sigma^G=\pi^G$. Since $r\in G$, we obtain $\sigma^G\in\tau^G$. The other cases are similar.
 
Finally, using \eqref{al:ft} we prove that $p\Vdash_\PP\varphi(\sigma,\tau)$ iff $p\leq_{\BB(\PP)}\llbracket\varphi(\sigma,\tau)\rrbracket$. 
Since $\leq_{\BB(\PP)}\in\C$, this will finish the proof of the theorem. 
Suppose that $p\Vdash_\PP\varphi(\sigma,\tau)$. If $p\nleq_{\BB(\PP)}\llbracket\varphi(\sigma,\tau)\rrbracket$, then by separativity of $\BB(\PP)$ and 
density of $\PP$ in $\BB(\PP)$ there is $q\leq_\PP p$ such that $q\bot_{\BB(\PP)}\llbracket\varphi(\sigma,\tau)\rrbracket$. 
Now let $G$ be $\PP$-generic with $q\in G$. Then $p$ is also in $G$, so by assumption $M[G]\models\varphi(\sigma^G,\tau^G)$. By \eqref{al:ft}
there is $r\in G$ with $r\leq_{\BB(\PP)}\llbracket\varphi(\sigma,\tau)\rrbracket$. But then $r$ and $q$ are compatible, and hence so are 
$q$ and $\llbracket\varphi(\sigma,\tau)\rrbracket$, which is impossible. The converse is an immediate consequence of \eqref{al:ft}.
\end{proof}

In particular, this proves Theorem \ref{theorem:BooleanCompletion}.

\begin{corollary}\label{cor:pretame bc}
 If $\MM$ is a model of $\GB$ then every separative antisymmetric pretame notion of class forcing for $\MM$ has a Boolean completion in $\MM$. 
\end{corollary}

\begin{proof}
In \cite[Theorem 2.18]{MR1780138}, it is shown that every pretame notion of class forcing satisfies the definability lemma
for atomic formulae. The corollary thus follows from Theorem \ref{thm:compl forcing thm}. 
\end{proof}

\begin{lemma}\label{lem:KM bc}
 If $\MM$ is a model of $\KM$, then every separative antisymmetric notion of class forcing for $\MM$ has a Boolean completion in $\MM$. 
\end{lemma}

\begin{proof}
Let $\MM=\langle M,\C\rangle$ be a model of $\KM$ and let $\PP=\langle P,\leq_\PP\rangle$ be a separative antisymmetric notion of class forcing for $\MM$. Making use of a suitable bijection in $\C$, we may assume that $P\cap\On^M=\emptyset$. Pick a disjoint partition $\langle A_\alpha\mid\alpha\in\On^M\rangle$ of $\On^M$ such that $\{\langle\alpha,\beta\rangle\mid\beta\in A_\alpha\}\in\C$. Using class recursion, we define $\subseteq$-increasing 
sequences $\langle\PP_\alpha\mid\alpha\in\On^M\rangle$ and $\langle\QQ_\alpha\mid\alpha\in\On^M\rangle$ of separative 
notions of class forcing containing $\PP$ such that $\{\langle p_0,p_1,\alpha\rangle\mid p_0\le_{\PP_\alpha}p_1\}\in\C$, $\{\langle q_0,q_1,\alpha\rangle\mid q_0\le_{\QQ_\alpha}q_1\}\in\C$
and $\PP$ is dense in each $\PP_\alpha^*$ and $\QQ_\alpha^*$, where $\PP_\alpha^*$ and $\QQ_\alpha^*$ are the notions of class forcing
obtained from $\PP_\alpha$ or $\QQ_\alpha$ respectively, by removing all conditions $p$ which stengthen every other condition (i.e.\ which  are equivalent to $0_{\PP_\alpha}$ or $0_{\QQ_\alpha}$ respectively).

Let $P_0=P$. If $\alpha$ is a limit ordinal, let $P_\alpha=\bigcup_{\beta<\alpha}P_\beta$ and $\leq_{\PP_\alpha}=\bigcup_{\beta<\alpha}\leq_{\PP_\beta}$ for every $\beta<\alpha$. 
Suppose that $\PP_\alpha$ has been defined. We construct $\QQ_\alpha$ by adding suprema for all subsets of $P_\alpha$ in $M$, and then construct $\PP_{\alpha+1}$ by adding negations for all elements of $Q_\alpha$.

More precisely, let $Q_\alpha=P_\alpha\cup\{\sup A\mid A\in M,A\subseteq P_\alpha\}$, where $\sup A\in A_{2\cdot\alpha}$ is different for each $A\in M$ with 
$A\subseteq P_\alpha$ and $\{\langle A,\sup A\rangle\mid A\in M,A\subseteq P_\alpha\}\in\C$. Thus $Q_\alpha\in\C$ and we define an ordering 
$\leq_{\QQ_\alpha}$ on $Q_\alpha$ with $\leq_{\QQ_\alpha}\cap\,(P_\alpha\times P_\alpha)=\,\leq_{\PP_\alpha}$ and $\leq_{\QQ_\alpha}\in\C$ in the following way:
\begin{equation}\label{al:sup}
 \begin{split}
\sup A\leq_{\QQ_\alpha} p\quad & \Llr\quad\forall a\in A(a\leq_{\PP_\alpha}p), \\
p\leq_{\QQ_\alpha}\sup A\quad & \Llr\quad A\text{ is predense below }p\text{ in }\PP_\alpha, \\ 
\sup A\leq_{\QQ_\alpha}\sup B\quad & \Llr\quad\forall a\in A(a\leq_{\QQ_\alpha}\sup B) 
 \end{split}
\end{equation}
for all $p\in\PP_\alpha$, and $A,B\in M$ with $A\subseteq P_\alpha$.  
Firstly, we check that $\PP$ is dense in $\QQ_\alpha^*$. If $A\subseteq P_\alpha$ with $A\in M$ such that $\sup A\in\QQ_\alpha^*$ then 
$A\cap\PP_\alpha^*\neq\emptyset$. Now if $a\in A$ is in $\PP_\alpha^*$ then $a$ strenghtens $\sup A$, and since by assumption $\PP$ is dense in $\PP_\alpha^*$,
there is $p\in\PP$ with $p\leq_{\PP_\alpha}a$. In particular, $p\leq_{\QQ_\alpha}\sup A$. 
In order to prove that $\QQ_\alpha$ is separative,
since $\sup\{p\}=p$ for $p\in\PP_\alpha$, it suffices to check that whenever $A,B\subseteq\PP_\alpha$ are sets 
with $\sup A\nleq_{\QQ_\alpha}\sup B$, then there is $p\leq_{\QQ_\alpha}\sup A$ incompatible with $\sup B$. 
So suppose that $\sup A\nleq_{\QQ_\alpha}\sup B$. Then there is $a\in A$ with $a\nleq_{\QQ_\alpha}\sup B$, i.e.\ 
there exists a strengthening $p\in\PP$ of $a$ such that each $q\leq_\PP p$ is incompatible with every element of $B$.
In particular, $p$ is incompatible with $\sup B$. 

Now let $\PP_{\alpha+1}=\QQ_\alpha\cup\{\neg q\mid q\in\QQ_\alpha\}$, where $\neg q\in A_{2\cdot\alpha+1}$ is different for each $q\in\QQ_\alpha$, such that $\{\langle q,\neg q\rangle\mid q\in\QQ_\alpha\}\in\C$. Thus $\PP_{\alpha+1}\in\C$ and we define an ordering $\leq_{\PP_{\alpha+1}}$ on $\PP_{\alpha+1}$
extending $\leq_{\QQ_\alpha}$ such that $\leq_{\PP_{\alpha+1}}\in\C$ as follows:
\begin{equation}\label{al:neg}
 \begin{split}
p\leq_{\PP_{\alpha+1}}\neg q\quad&\Llr\quad p\bot_{\QQ_\alpha}q, \\ 
\neg p\leq_{\PP_{\alpha+1}} q\quad&\Llr\quad\forall r\in\QQ_\alpha\ (r\parallel_{\QQ_\alpha} p\vee r\parallel_{\QQ_\alpha} q), \\ 
\neg p\leq_{\PP_{\alpha+1}}\neg q\quad&\Llr\quad q\leq_{\QQ_\alpha} p \end{split}
\end{equation}
for all $p,q\in\QQ_\alpha$. 
Again, we need to verify that $\PP$ is dense in $\PP_{\alpha+1}^*$. Let $q\in\QQ_\alpha$ be a condition such that $\neg q$ is non-zero, i.e.\ there is $p\in\PP_{\alpha+1}$
such that $\neg q\nleq_{\PP_{\alpha+1}} p$. If $p\in\QQ_\alpha$, then this means that there is $r\in\QQ_\alpha$ which is incompatible
with $p$ and $q$ in $\QQ_\alpha$. In particular, $r\leq_{\PP_{\alpha+1}}\neg q$ and since $\PP$ is dense in $\QQ_\alpha$, 
we can strengthen $r$ to some condition in $\PP$. Otherwise, $p$ is of the form $\neg r$ for some 
$r\in\QQ_\alpha$. Therefore, $r\nleq_{\QQ_\alpha}q$ and so by separativity of $\QQ_\alpha$ there is $s\in\QQ_\alpha$ 
with $s\leq_{\QQ_\alpha}r$ and $s\bot_{\QQ_\alpha}q$. But then $s\leq_{\PP_{\alpha+1}}\neg q$ and once again we apply the density 
of $\PP$ in $\QQ_\alpha$ to obtain the result. 

Secondly, we need to check that separativity is preserved. First suppose that 
$p,q\in\QQ_\alpha$ such that $p\nleq_{\PP_{\alpha+1}}\neg q$. 
Then $p$ and $q$ are compatible, hence there is 
$r\in\QQ_\alpha$ with $r\leq_{\QQ_\alpha}p,q$. In particular, $r$ and $\neg q$ are incompatible in $\PP_{\alpha+1}$. 
If $\neg p\nleq_{\PP_{\alpha+1}}q$, then there is $r\in\PP$ such that $r\bot_{\QQ_\alpha}p$ and $r\bot_{\QQ_\alpha}q$.
But this means that $r\leq_{\PP_{\alpha+1}}\neg p,\neg q$ and so $r$ and $q$ are incompatible in $\PP_{\alpha+1}$. 
Finally, suppose that $\neg p\nleq_{\PP_{\alpha+1}}\neg q$. 
Then $q\nleq_{\QQ_\alpha}p$, so by separativity of $\QQ_\alpha$ there is a strengthening $r\in\QQ_\alpha$ of $q$ with 
$r\bot_{\QQ_\alpha}p$. This means that $r\leq_{\PP_{\alpha+1}}\neg p$ and clearly $r$ is incompatible with $\neg q$. 

Then $\bigcup_{\alpha\in\On^M}\PP_\alpha=\bigcup_{\alpha\in\On^M}\QQ_\alpha$ is in $\C$. However, the order given 
by $p\leq q$ if and only if there is $\alpha\in\On^M$ such that $p,q\in\PP_\alpha$ and $p\leq_{\PP_\alpha}q$ is
only a preorder. To obtain a partial order, consider the equivalence relation 
$$p\approx q\quad\Llr\quad\exists\alpha\in\On^M(p,q\in\PP_\alpha\wedge p\leq_{\PP_\alpha}q\wedge q\leq_{\PP_\alpha}p).$$
For $p\in\bigcup_{\alpha\in\On}\PP_\alpha$, we define 
$$[p]=\{q\mid q\approx p\,\land\,\forall r\,[r\approx p\ra\rank(r)\geq\rank(q)]\}\in M$$
and $\BB(\PP)=\{[p]\mid\exists\alpha\in\On^M(p\in\PP_\alpha)\}$.
Now we can introduce the ordering, suprema, infima and complements in $\BB(\PP)$ in the canonical way and let
$\one_{\BB(\PP)}=[\one_\PP]$ as well as $0_{\BB(\PP)}=\neg\one_{\BB(\PP)}$. More precisely, if $A\subseteq\BB(\PP)$ is a set, 
then let $\alpha\in\On^M$ be least such that every member of $A$ has a representative in $(\V_\alpha)^M$ and $\bar A=\{p\in(\V_\alpha)^M\mid\exists\beta\in\On^M(p\in\PP_\beta\wedge[p]\in A)\}$.
Then we can define $\sup A=[\sup_{\QQ_\beta}\bar A]$, where $\beta$ is the least ordinal with $\bar A\subseteq\PP_\beta$. It follows directly from \eqref{al:sup} that $\sup A$ is well-defined and that it is the supremum of $A$ in $\BB(\PP)$. 
Complements are defined in the same way, i.e.\ for $[p]\in\BB(\PP)$ with $p\in\QQ_{\alpha}$ put $\neg[p]=[\neg p]$, 
where $\neg p$ is defined in $\PP_{\alpha+1}$. Again, it is a straightforward consequence of \eqref{al:neg} and 
the density of $\PP$ in $\QQ_{\alpha}^*$ that this is well-defined and that this actually defines the complement of $[p]$ in $\BB(\PP)$. 
Furthermore, if $A\in M$ is a subset of $\BB(\PP)$, then we let $\inf A=\neg(\sup\{\neg a\mid a\in A\})$.  
Moreover, the embedding $\pi:\PP\ra\BB(\PP)\setminus\{0_{\BB(\PP)}\},p\mapsto[p]$ is dense by construction and 
it follows from $\PP$ being separative and antisymmetric that $\pi$ is injective. 
\end{proof}

Note that Robert Owen has erroneously used a similar proof in \cite{owen2008outer} to show that in $\ZFC$ every separative antisymmetric notion of class forcing 
has a Boolean completion. However in Section \ref{sec:fr not def} we will show that the definability lemma
for atomic formulae can fail in $\ZF^-$ (and also in $\ZFC$), and hence Boolean completions do not always exist. 

Regarding separativity, if the $\GB^-$-model $\MM$ satisfies representatives choice, one can build a separative quotient for any notion of class forcing.
Recall that the separative quotient is obtained by considering the equivalence relation
\begin{align*}
 p\approx q\quad\text{if and only if}\quad\forall r\in\PP\:(r\parallel_\PP p\lr r\parallel_\PP q).
\end{align*}
Unlike in set forcing the equivalence classes can in fact be proper classes. 
Using representatives choice, there is a surjective map $\pi:\PP\ra\Ss(\PP)$ with $\pi,\Ss(\PP)\in\C$ such that 
$\pi(p)=\pi(q)$ if and only if $p\approx q$. 
Furthermore, we equip $\Ss(\PP)$ with the usual ordering given by 
$$\pi(p)\leq_{\Ss(\PP)}\pi(q)\quad\text{if and only if}\quad\forall r\in\PP(r\parallel_\PP p\ra r\parallel_\PP q).$$
It is straightforward to check that the ordering $\leq_{\Ss(\PP)}$ is well-defined. 
For $\sigma\in M^\PP$, define recursively
$$\sigma^\pi=\{\langle\tau^\pi,\pi(p)\rangle\mid\langle\tau,p\rangle\in\sigma\}.$$
It is easy to see that as in set forcing, $\PP$ and $\Ss(\PP)$ generate the same generic extensions and
$$p\Vdash_\PP\varphi(\sigma_0,\dots,\sigma_{n-1})~\Llr~\pi(p)\Vdash_{\Ss(\PP)}\varphi(\sigma_0^\pi,\dots,\sigma_{n-1}^\pi).$$
The following provides an alternative proof of \cite[Lemma 11, Lemma 12]{antos2015class}.

\begin{corollary}\label{thm:KM bc}
 If $\MM$ is a model of $\KM$, every notion of class forcing satisfies the forcing 
theorem for all $\L^n$-formulae over $\MM$.
\end{corollary}
\begin{proof}
Since by $\KM$ there is a global well-ordering of $\MM$, every notion of class forcing $\PP$ of $\MM$ has a separative quotient $\Ss(\PP)$ over $\MM$. 
Therefore, Lemma \ref{lem:KM bc} implies that $\Ss(\PP)$ has a Boolean completion. By Theorem \ref{thm:compl forcing thm}, 
$\Ss(\PP)$ satisfies the forcing theorem for all $\L_\in$-formulae. Now by the above observations, this implies that 
$\PP$ satisfies the forcing theorem for all $\L_\in$-formulae.
\end{proof}

\section{Approachability by projections}\label{section:Approach}
As usual, we fix a countable transitive model $\MM=\langle M,\C\rangle$ of $\GB^-$. We define a fairly weak combinatorial condition on notions of class forcing that implies the forcing theorem to hold. In particular, this property is 
satisfied by the forcing notions $\Col(\omega,\On)^M$, $\Col_*(\omega,\On)^M$ and $\Col_\ge(\omega,\On)^M$ from Section \ref{examples}.

\begin{definition}\label{appproj}
We say that a class forcing $\PP=\langle P,\leq_\PP\rangle$ for $\MM$ is \emph{approachable by projections} if it can be written as 
a continuous, increasing union $\PP=\bigcup_{\alpha\in\On^M}\PP_\alpha$ for a sequence $\langle \PP_\alpha\mid\alpha\in\On\rangle\in\C$ of notions 
of set forcing $\PP_\alpha=\langle P_\alpha,\le_{\PP_\alpha}\rangle$, where $\le_{\PP_\alpha}$ is the ordering on $P_\alpha$ induced by $\PP$, for which there exists a sequence of maps $\langle\pi_{\alpha+1}\mid\alpha\in\On^M\rangle$ so that 
$\pi_{\alpha+1}\colon P\to P_{\alpha+1}$, 
$\{\langle\alpha,p,\pi_{\alpha+1}(p)\rangle\mid\alpha\in\On^M,p\in P\}\in\C$
 and for every $\alpha\in\On^M$, the following hold:
\begin{enumerate}
  \item $\pi_{\alpha+1}(\one_\PP)=\one_\PP$,
  \item $\forall p,q\in P\ (p\leq_\PP q\to\pi_{\alpha+1}(p)\leq_{\PP}\pi_{\alpha+1}(q))$,
  \item $\forall p\in P\,\forall q\leq_{\PP_{\alpha+1}}\pi_{\alpha+1}(p)\,\exists r\leq_\PP p\ ( \pi_{\alpha+1}(r)\leq_{\PP} q)$,
  \item $\forall p\in P_\alpha\,\forall q\in P\ (\pi_{\alpha+1}(q)\leq_{\PP} p\to q\leq_\PP p)$ and
  \item $\pi_{\alpha+1}$ is the identity on $P_{\alpha}$.
\end{enumerate}
\end{definition}
Note that (in order to justify our terminology), each $\pi_{\alpha+1}$ is in particular (by (1)--(3)) a projection from $\PP$ to $\PP_{\alpha+1}$. It follows that each $\pi_{\alpha+1}$ is a dense embedding and thus $\pi_{\alpha+1}''G$ is a $\PP_{\alpha+1}$-generic filter whenever $G$ is a $\PP$-generic filter.

\begin{lemma}\label{lem:proj gen ext}
If $\mathbb{P}=\bigcup_{\alpha\in\On^M}\PP_\alpha$ is approachable by projections with projections 
$\pi_{\alpha+1}:\PP\ra\PP_{\alpha+1}$ and $G$ is $\PP$-generic, then $M[G]\subseteq\bigcup_{\alpha\in\On^M}M[\pi_{\alpha+1}''G]$, and the latter is a union of set-generic extensions of $M$.
\end{lemma}

\begin{proof}
If $\sigma$ is a $\PP$-name, then there is $\alpha\in\On^M$ such that $\sigma$ is already a $\PP_\alpha$-name. 
Since $\pi_{\alpha+1}$ is dense, $G_{\alpha+1}=\pi_{\alpha+1}''G$ is $\PP_{\alpha+1}$-generic. By property (5) of Definition
\ref{appproj}, $\sigma^G=\sigma^{G_{\alpha+1}}\in M[G_{\alpha+1}]$. 
\end{proof}


\begin{lemma}\label{lem:col proj}
  $\Col(\omega,\On)^M$, $\Col_*(\omega,\On)^M$ and $\Col_{\ge}(\omega,\On)^M$ are approachable by projections.
\end{lemma}
\begin{proof}
Let $\PP=\Col(\omega,\On)^M$ and let $\PP_\alpha=\Col(\omega,\alpha)$. Furthermore, take $\pi_{\alpha+1}$ to be the map that, for $p\in\Col(\omega,\On)^M$, replaces the value of
  $p(n)$ by $\alpha$ whenever $p(n)>\alpha$. Conditions (1), (2) and (5) of Definition \ref{appproj} are trivially satisfied.
  For (3) let $p\in \Col(\omega,\On)^M$ and $q\in \Col(\omega,\alpha+1)$ such that 
  $q\leq_{\PP_{\alpha+1}} \pi_{\alpha+1}(p)$. Then $r=p\cup\{\langle n,q(n)\rangle\mid n\in\dom q\setminus \dom p\}$
  satisfies $r\leq_\PP p$ and $\pi_{\alpha+1}(r)\leq_\PP q$.
  For (4), consider $p\in \Col(\omega,\alpha)$ and $q\in \Col(\omega,\On)^M$ such that 
  $\pi_{\alpha+1}(q)\leq_{\PP} p$. Suppose $n\in\dom p\cap\dom q$.
  Since $p\in \Col(\omega,\alpha)$, $p(n)<\alpha$
  and hence $q(n)=\pi_{\alpha+1}(q)(n)=p(n)$. Thus $q\leq_\PP p$. 
The arguments for $\Col_*(\omega,\On)^M$ and $\Col_{\ge}(\omega,\On)^M$ are similar. 
\end{proof}

\begin{theorem}\label{thm:proj forcing thm}
  If $\PP$ is approachable by projections, then the forcing relation for ``$v_0=v_1"$ is definable. Therefore, by Theorem \ref{thm:ft}, $\PP$ satisfies the forcing theorem for every $\mathcal L^n$-formula.
\end{theorem}
\begin{proof}
Fix any ordinal $\alpha\in M$. We will show by induction on name rank that the forcing relation for ``$v_0\subseteq v_1"$,
restricted to names that only mention conditions in $P_\alpha$, is definable. Then the forcing relation
for $\anf{v_0=v_1}$ is definable by $p\Vdash_\PP\sigma=\tau$ if and only if $p\Vdash_\PP\sigma\subseteq\tau$ and $p\Vdash_\PP\tau\subseteq\sigma$.
It will be easy to see that this is uniform in $\alpha$, thus implying the desired statement. 

For $\sigma,\tau\in M^{\PP_\alpha}$ and $p\in P$ define $p\Vdash_\PP^*\sigma\subseteq\tau$ if and only if 
\begin{align*}
\forall\langle\rho,s\rangle\in\sigma\forall q\leq_\PP p\exists r\leq_\PP q(r\leq_\PP s\ra\exists\langle\pi,t\rangle\in\tau(r\leq_\PP t\wedge r\Vdash_{\PP}^*\rho=\tau)).
\end{align*}
Similiarly, let $\pi_{\alpha+1}(p)\Vdash_\PP^{*,\alpha+1}\sigma\subseteq\tau$ denote the same formula as above but where $p$ is replaced 
by $\pi_{\alpha+1}(p)$, $\Vdash_\PP^*$ is replaced by $\Vdash_\PP^{*,\alpha+1}$ and all quantifiers over $P$ are restricted to $P_{\alpha+1}$. Furthermore, $p\Vdash_\PP^*\sigma=\tau$ is an
abbreviation for $p\Vdash_\PP^*\sigma\subseteq\tau$ and $p\Vdash_\PP^*\tau\subseteq\sigma$ and similarly for $\Vdash_\PP^{*,\alpha+1}$. 

We will show by induction on the rank of $\PP_\alpha$-names that 
\begin{align}
 p\Vdash_\PP^*\sigma\subseteq\tau\quad\text{if and only if}\quad\pi_{\alpha+1}(p)\Vdash_\PP^{*,\alpha+1}\sigma\subseteq\tau.\tag{$\ast$}
\end{align}
The right-hand side is clearly definable with parameter $\alpha$, since $\PP_{\alpha+1}$ is a set forcing. Moreover, 
by the usual proof, $p\Vdash_\PP^*\sigma\subseteq\tau$ if and only if $p\Vdash_\PP\sigma\subseteq\tau$. This shows 
that assuming ($\ast$), the forcing relation for ``$v_0\subseteq v_1"$ restricted to $\PP_\alpha$-names is definable. 

Assume first that $p\Vdash_\PP\sigma\subseteq\tau$ and let $\langle\rho,s\rangle\in\sigma$ and 
$\bar{q}\leq_{\PP_{\alpha+1}}\pi_{\alpha+1}(p)$. By (3) there is 
$q\leq_\PP p$ such that $\pi_{\alpha+1}(q)\leq_\PP\bar{q}$. By assumption there is $r\leq_\PP q$ witnessing 
$p\Vdash_\PP^*\sigma\subseteq\tau$. 
Using (2), $\pi_{\alpha+1}(r)\leq_\PP\bar{q}$. Now if $r\nleq_\PP s$, then $\pi_{\alpha+1}(r)\nleq_\PP s$ by (4).
Otherwise, assume that $\langle\pi,t\rangle\in\tau$
such that $r\leq_\PP t$ and $r\Vdash^*_\PP\rho=\pi$. Again by (2) and (5) we have $\pi_{\alpha+1}(r)\leq_\PP t$ and inductively,
$\pi_{\alpha+1}(r)\Vdash_\PP^{*,\alpha+1}\rho=\pi$. 

For the converse, suppose $\pi_{\alpha+1}(p)\Vdash_\PP^{*,\alpha+1}\sigma\subseteq\tau$ and let $\langle\rho,s\rangle\in\sigma$.
 Let $q\leq_\PP p$. By (2) we have $\pi_{\alpha+1}(q)\leq_\PP\pi_{\alpha+1}(p)$ and 
thus there is $\bar{r}\leq_{\PP_{\alpha+1}}\pi_{\alpha+1}(q)$ witnessing $\pi_{\alpha+1}(p)\Vdash_\PP^{*,\alpha+1}\sigma\subseteq\tau$. 
Using (3), choose $r\leq_\PP q$ such that $\pi_{\alpha+1}(r)\leq_\PP\bar{r}$. 
Now if $\bar{r}\bot_{\PP_{\alpha+1}} s$, then 
also $r\bot_\PP s$ because if $t\leq_\PP r,s$, then $\pi_{\alpha+1}(t)\leq_\PP\bar{r},s$ by (2) and (5). 
So assume that $\bar r$ and $s$ are compatible and take $\bar{u}\in P_{\alpha+1}$ such that $\bar u\leq_\PP\bar r,s$.
Hence again using that $\pi_{\alpha+1}(p)\Vdash_\PP^*\sigma\subseteq\tau$ there is $\bar{v}\leq_{\PP_{\alpha+1}}\bar{u}$ 
such that $\bar v\leq_\PP s$ and there is $\langle\pi,t\rangle\in\tau$ such that $\bar{v}\leq_\PP t$ and 
$\bar{v}\Vdash_\PP^{*,\alpha+1}\rho=\pi$. Now since $\bar{v}\leq_{\PP_{\alpha+1}}\pi_{\alpha+1}(q)$ there exists
$v\leq_\PP q$ such that $\pi_{\alpha+1}(v)\leq_\PP\bar v$, so $\pi_{\alpha+1}(v)\Vdash_\PP^{*,\alpha+1}\rho=\tau$. 
Then by (4) we get that $v\leq_\PP s,t$ and inductively, $v\Vdash_\PP^*\rho=\tau$.
\end{proof}

The following lemma shows that the converse does not hold in general. 

\begin{lemma}Suppose that $M$ is a countable transitive model of $\ZFC$. 
There is a tame notion of class forcing which is not approachable by projections. 
\end{lemma} 
\begin{proof}
Let $\PP$ be Jensen coding, as described in \cite{MR645538}. Then $\PP$ is a tame notion of forcing, i.e.\ it preserves $\ZFC$. Assume that $M\models\GCH$ (otherwise
ensure this by previously forcing $\GCH$).
Then the extension of $M$ is 
$M[G]=\mathsf L[x]$ for some real $x$ which is not contained in any set forcing extension of $M$, so by Lemma
\ref{lem:proj gen ext}, $\PP$ cannot be approachable by projections.  
\end{proof} 

The next lemma shows that approachability by projections is essentially a weakening of being an increasing union of set-sized complete subforcings. 

\begin{lemma}\label{lem:union proj}
  If $\PP=\bigcup_{\alpha\in\On^M}\PP_\alpha$ is an increasing union of set-sized complete subforcings (as witnessed by a class in $\C$), 
then there is a dense embedding $i\in\C$, $i:\PP\ra\BB$ from $\PP$ 
into an $M$-complete Boolean algebra $\BB\in\C$ which is approachable by projections. Moreover, $\PP$ and 
$\BB$ have the same generic extensions, i.e.\ whenever $G$ is $\BB$-generic over $\MM$, then 
$M[G]=M[i^{-1}[G]]$. 
\end{lemma}

\begin{proof}
By passing to separative quotients,
we can assume that every $\PP_\alpha$ is separative and antisymmetric.
Since each $\PP_\alpha$ is a set forcing, it has a Boolean 
completion $\BB(\PP_\alpha)$ given by the set of all regular open subsets of $\PP_\alpha$ and there is a dense
embedding $e_\alpha:\PP_\alpha\ra\BB(\PP_\alpha)$. By defining suitable 
embeddings $i_{\alpha\beta}$ from $\BB(\PP_\alpha)$ into $\BB(\PP_\beta)$ and then forming the quotient 
of $\bigcup_{\alpha\in\On^M}\BB(\PP_\alpha)$ modulo the equivalence relation 
$b_0\sim b_1$ iff $b_1=i_{\alpha\beta}(b_0)$ if $b_0\in\BB(\PP_\alpha)$ and $b_1\in\BB(\PP_\beta)$ and $\alpha<\beta$ or $b_0=i_{\alpha\beta}(b_1)$ if $b_0\in\BB(\PP_\alpha)$ 
and $b_1\in\BB(\PP_\beta)$ and $\alpha\geq\beta$, we obtain a Boolean algebra $\BB=\bigcup_{\alpha\in\On^M}\BB_\alpha$, where
every $\BB_\alpha=\{[b]\mid b\in\BB(\PP_\alpha)\}$ is a complete subforcing of $\BB$. Note that the equivalence classes are, in general, proper classes. We can avoid this problem
by considering \begin{align*}
 [b]=\{c\in\bigcup_{\alpha\in\On}\BB(\PP_\alpha)\mid b\sim c\wedge\exists\alpha\in\On^M\,[c\in\BB(\PP_\alpha)\wedge\neg\exists\beta<\alpha\,\exists d\in\BB(\PP_\beta)\ (d\sim b)]\}.
\end{align*}
Define, for every $\alpha\in\On^M$, the projection $\pi_{\alpha+1}$ by setting
$$\pi_{\alpha+1}([b]) ~= ~\sup\{[c]\in\BB_{\alpha+1}\mid [c]\leq [b]\}$$ for every $[b]\in\BB$.
Straightforward calculations yield those projections to witness that $\BB$ is approachable by projections.
Moreover, $i:\PP\ra\BB, i(p)=[e_\alpha(p)]$ for $p\in P_\alpha$ is a dense embedding. By standard computations using 
that every $\PP_\alpha$ is a complete subforcing of $\PP$ and that the analogous property holds for $\BB_\alpha$ and $\BB$, it holds
that for every $\BB$-generic filter $G$, $M[G]=M[i^{-1}[G]]$ as desired. 
\end{proof}

\begin{lemma}
 There is a notion of class forcing which is an increasing union of set-sized complete subforcings but not pretame. 
\end{lemma}

\begin{proof}
By Lemma \ref{lemma:properties col}, the forcing notion $\Col_\geq(\omega,\On)^M$ is an increasing union of set-sized
complete subforcings, but it is not pretame since it adds a cofinal function from $\omega$ to the ordinals. 
\end{proof}

If we combine Lemma \ref{lem:union proj} and Theorem \ref{thm:proj forcing thm} with the results from the 
previous section, we obtain that every union of set-sized complete subforcings has a Boolean completion 
and satisfies the forcing theorem for all $\L^n$-formulae. As an application we obtain that 
any $\On^M$-length iteration and product of set forcing notions in $M$ satisfies the forcing theorem for $\L^n$-formulae. 
Note however that the fact that unions of set-sized complete subforcings satisfy the forcing theorem for $\L_\in$-formulae has 
already been proven by Zarach in \cite{MR0345819}.


\section{Failures of the definability lemma}\label{sec:fr not def}

The main goal of this section is to show that for every countable transitive model $M$ of $\ZF^-$, the forcing relation of $\FF^M$ (as defined in Section \ref{examples}) is not first-order definable over $\langle M,\text{Def}(M)\rangle$. 
Furthermore, we will prove that for certain models $M$ of $\ZFC$ its forcing relation is not $M$-amenable. 
Whenever it is clear from context which model is referred to, we write $\FF$ for $\FF^M$.  
Unless stated otherwise, $\MM=\langle M,\C\rangle$ will denote an arbitrary countable transitive model of $\GB^-$. 

Following {\cite[Chapter 3.5]{drake1974set}}, we let $\Fm\subseteq{}^{{<}\omega}\omega$ denote the set of all codes for $\mathcal{L}_\in$-formulae.  
Since we work inside some model $\V$ of set theory and we use these codes inside
countable transitive models that are elements of $\V$ together with the corresponding formalized satisfaction relation, 
we may assume that each  element of $\Fm$ is the G\"odel number $\gbr{\varphi}$ of an $\mathcal{L}_\in$-formula $\varphi$. 
For $k\in\omega$, let $\Fm_k$ denote the set of all G\"odel numbers for formulae with free variables among $\{v_0,\dots,v_{k-1}\}$. For the sake of simplicity, we will assume that every $\L_\in$-formula $\varphi$ is in the following normal form:
Whenever $\exists v_k\psi$ is a subformula of $\varphi$, then the free variables of $\psi$ are among $\{v_0,\dots,v_k\}$. 

\begin{definition} 
 A relation $T\subseteq\Fm_1\times M$ is a \emph{first-order truth predicate for $M$} if $$\langle\gbr{\varphi},x\rangle\in T ~ \Longleftrightarrow ~ \langle M,\in\rangle\models\varphi(x)$$ holds for every $\gbr{\varphi}\in\Fm_1$ and every $x\in M$.  
\end{definition} 


Let $G$ be an $\FF$-generic filter over $\MM$ and let $E$ and $F$ be defined as in the proof of Lemma \ref{lemma:friedman forcing}. Then $$T ~ = ~ \Set{\langle\gbr{\varphi},x\rangle\in\Fm_1\times M}{\langle \omega,E\rangle\models\varphi(F^{{-}1}(x))}\subseteq M$$ is a first-order truth predicate for $\MM$ and, by Tarski's Undefinability Theorem, $T$ cannot be defined over $M$ by a first-order formula. In the following, we will show that definability of the forcing relation for $\FF$ would lead to a first-order definition of $T$. 


\begin{notation}
If $\vec x=x_0,\dots,x_{k-1}$ is a sequence in $M$, we say that a sequence $\vn=n_0,\dots,n_{k-1}$ in $\omega$ is
\emph{appropriate for $\vec x$}, if for all $i,j<k$, $x_i=x_j$ if and only if $n_i=n_j$. We inductively define $p_{\vn}^{\vec x}\in\FF$ as follows, whenever
$\vn$ is a sequence of natural numbers which is appropriate for $\vec x$.
\begin{enumerate-(1)}
 \item If $k=0$, then $p_\emptyset^\emptyset=\one_\FF$.
 \item If $\vn,n_k$ is appropriate for $\vec x,x_k$, given $p=p_{\vn}^{\vec x}$, let
$p_{\vn,n_k}^{\vec x,x_k}$ be the condition $q\in\FF$ with domain $d_q=d_p\cup\{n_k\}$,
$f_q=f_p\cup\{\langle n_k,x_k\rangle\}$ and 
$$e_q=e_p\cup\{\langle n_k,n_i\rangle\mid i\in\vec n\,\land\,x_k\in x_i\}\cup\{\langle n_i,n_k\rangle\mid i\in\vec n\,\land\,x_i\in x_k\}.$$
\end{enumerate-(1)}
Clearly, we obtain that whenever $\vec x$ extends $\vec y$, $\vn$ extends $\vm$ and $\vn$ is appropriate for $\vec x$, then 
$p_{\vn}^{\vec x}\leq_\FF p_{\vm}^{\vec y}$. Furthermore, we define $p^{\vec x}$ to be the condition $p_{\vn}^{\vec x}$,
where $\vn$ is the lexicographically smallest sequence which is appropriate for $\vec x$. 

\end{notation}

Before we proceed to prove that the definability lemma can fail for $\FF$, we need a translation from 
$\L_\in$-formulae to $\L_{\omega_1,0}^{\Vdash}(\FF,M)$-formulae so that we can apply Theorem \ref{thm:compl forcing thm}, where
$\L_{\omega_1,0}^{\Vdash}(\FF,M)$-formulae are $\L_{\On,0}^{\Vdash}(\FF,M)$-formulae in which all conjunctions and disjunctions are countable.

\begin{notation}
Inductively, we assign to every $\L_\in$-formula $\varphi$ with free variables in $\{v_0,\dots,v_{k-1}\}$ and 
all sequences $\vn=n_0,\dots,n_{k-1}$ of natural numbers an $\L_{\omega_1,0}^{\Vdash}(\FF,M)$-formula $\varphi_{\vn}^*$ as follows:
\begin{align*}
 (v_i=v_j)_{\vn}^*&=(\check{n_i}=\check{n_j})\\
 (v_i\in v_j)_{\vn}^*&=(\op(\check{n_i},\check{n_j})\in\dot E)\\
 (\neg\varphi)_{\vn}^*&=(\neg\varphi_{\vn}^*)\\
 (\varphi\vee\psi)_{\vn}^*&=(\varphi_{\vn}^*\vee\psi_{\vn}^*)\\
 (\exists v_k\varphi)_{\vn}^*&=(\bigvee_{i\in\omega}\varphi_{\vn,i}^*).
\end{align*}
If $\vn=0,\dots,k-1$, then we simply write $\varphi^*$ for $\varphi_{\vn}^*$ and if $\vec x$ is a sequence in $M$ and 
$\vn$ is such that $p^{\vec x}=p_{\vn}^{\vec x}$, then we write $\varphi_{\vec x}^*$ for $\varphi_{\vn}^*$. In particular, 
if $v_0$ is the only free variable of $\varphi$, then $\varphi_x^*$ is $\varphi^*$. 
\end{notation}

The next lemma is the key ingredient to obtain a first-order truth predicate $T$ for $\MM$. We will use the 
translation of $\L_\in$-formulae to $\L_{\omega_1,0}^{\Vdash}(\FF,M)$-formulae to define truth by 
$\langle M,\in\rangle\models\varphi(x)$ if and only if $\langle\omega,E\rangle\models\varphi(n)$, where $n=F(x)$, 
if and only if $p_x\Vdash_\FF\varphi_x^*$. 

\begin{lemma}\label{lem:varphi*}
 For every $\L_\in$-formula $\varphi$ with free variables among $\{v_0,\dots,v_{k-1}\}$ and for all 
$\vec x=x_0,\dots,x_{k-1}\in M$, the following conditions hold:
\begin{enumerate-(1)}
 \item $\MM\models\varphi(\vec x)$ if and only if $p^{\vec x}\Vdash^\MM_\FF\varphi_{\vec x}^*$.
 \item $\MM\models\neg\varphi(\vec x)$ if and only if $p^{\vec x}\Vdash^\MM_\FF\neg\varphi_{\vec x}^*$.
\end{enumerate-(1)}
\end{lemma}

\begin{proof}
First, we verify that for every formula
$\varphi$ with free variables among $\{v_0,\dots,v_{k-1}\}$, and for all $\vn\in\omega^k$ appropriate for $\vec x$,
\begin{align}
 p^{\vec x}\Vdash_\FF\varphi_{\vec x}^* ~\Longleftrightarrow ~ p_{\vn}^{\vec x}\Vdash_\FF\varphi_{\vn}^*.\tag{$\ast$}\label{al:p^x}
\end{align}
Let $p^{\vec x}\Vdash_\FF\varphi_{\vec x}^*$ and let $\vec m$ be such that $p^{\vec x}=p_{\vec m}^{\vec x}$. Consider the automorphism $\pi$ on $\FF$ that for every condition
$p=\langle d_p,e_p,f_p\rangle$ replaces every $m_i$ appearing in $d_p,e_p$ and $\dom{f_p}$ by $n_i$. 
Clearly, $\pi(p^{\vec x})=p_{\vn}^{\vec x}$ and so $p_{\vn}^{\vec x}\Vdash_\FF\varphi_{\vn}^*$. The converse 
follows in the same way.

Working in $\V$, we now verify (1) and (2) by induction on the complexity of formulae. 
Observe that it suffices to check only that $\MM\models\varphi(\vec x)$ implies $p^{\vec x}\Vdash_\FF\varphi_{\vec x}^*$ and 
that $\MM\models\lnot\varphi(\vec x)$ implies $p^{\vec x}\Vdash_\FF\lnot\varphi_{\vec x}^*$, since the backwards directions of (1) and (2) immediately follow from the forward directions of (2) and (1) respectively.

For equations this is obvious. Suppose now that $\MM\models x\in y$. 
Let $G$ be generic over $\MM$ with $p^{x,y}\in G$. Then by definition of $p^{x,y}$, $\langle0,1\rangle\in E$ implying 
that $\MM[G]\models((v_0\in v_1)^*)^G$. The converse is similar. 

For negations, both (1) and (2) follow directly from the induction hypothesis.

We turn to disjunctions.
Assume that $\MM\models(\varphi\vee\psi)(\vec x)$. Without loss of generality,
assume that $\MM\models\varphi(\vec x)$. Then inductively, we get that 
$p^{\vec x}\Vdash_\FF\varphi_{\vec x}^*$. But then clearly $p^{\vec x}\Vdash_\FF(\varphi\vee\psi)_{\vec x}^*$.
Conversely, assume that $\MM\models\neg(\varphi\vee\psi)(\vec x)$. This means that 
$\MM\models\neg\varphi(\vec x)$ and $\MM\models\neg\psi(\vec x)$. By assumption, this means
that $p^{\vec x}\Vdash_\FF\neg\varphi_{\vec x}^*\wedge\neg\psi_{\vec x}^*$, hence $p^{\vec x}\Vdash_\FF\neg(\varphi\vee\psi)_{\vec x}^*$.

Suppose now that $\MM\models\exists v_k\varphi(\vec x,v_k)$. Then there 
is $y\in M$ such that $\MM\models\varphi(\vec x,y)$. This means that 
$p^{\vec x,y}\Vdash_\FF\varphi_{\vec x,y}^*$. Let $\vn$ be the sequence such that $p^{\vec x}=p_{\vn}^{\vec x}$. Now observe that by \eqref{al:p^x}, we have for every $i\in\omega$ 
such that $\vn,i$ is appropriate for $\vec x,y$ that
$p_{\vn,i}^{\vec x,y}\Vdash_\FF\varphi_{\vn,i}^*$. Take an $\FF$-generic filter $G$ over $\MM$ with $p^{\vec x}=p_{\vn}^x\in G$. 
By a density argument, there is $i\in\omega$ such that $\vn,i$ is appropriate for $\vec x,y$ and $p_{\vn,i}^{\vec x,y}\in G$. By assumption, this 
implies that $\MM[G]\models(\varphi_{\vn,i}^*)^G$, hence also $\MM[G]\models((\exists v_k\varphi)_{\vec x}^*)^G$. 

Assume now that $\MM\models\neg\exists v_k\varphi(\vec x,v_k)$. Then for every $y\in M$,
$\MM\models\neg\varphi(\vec x,y)$. Let $\vec n$ be the sequence in $\omega^k$ with $p^{\vec x}=p_{\vn}^{\vec x}$. 
We have to show that $p^{\vec{x}}\Vdash\neg\bigvee_{i\in\omega}\varphi_{\vn,i}^*$. Let $G$ be $\FF$-generic over $\MM$
with $p^{\vec x}\in G$ and suppose for a contradiction that $\MM[G]\models(\bigvee_{i\in\omega}\varphi_{\vn,i}^*)^G$.
Then there is $i\in\omega$ such that $\MM[G]\models(\varphi_{\vn,i}^*)^G$. Furthermore, there must 
be some $y\in M$ such that $p_{\vn,i}^{\vec x,y}\in G$. However, since $\MM\models\neg\varphi(\vec x,y)$, 
$p^{\vec x,y}\Vdash_\FF\neg\varphi_{\vec x,y}^*$ and therefore by \eqref{al:p^x}, $p_{\vn,i}^{\vec x,y}\Vdash_\FF\neg\varphi_{\vn,i}^*$ which 
is absurd. 
\end{proof}

For the rest of this section, we will assume, without loss of generality, that whenever $\varphi$ has exactly one free variable $v_i$, then $i=0$. 

\begin{theorem}\label{thm:ft truth pred}
If $\FF$ satisfies the definability lemma for $\anf{v_0\in v_1}$ or for $\anf{v_0=v_1}$ over $\MM$, then $\C$ contains a first-order 
truth predicate for $M$.
\end{theorem}

\begin{proof}
If the definability lemma holds either for $\anf{v_0\in v_1}$ or for $\anf{v_0=v_1}$, then $\FF$ satisfies the uniform forcing theorem 
for $\L_{\On,0}^{\Vdash}(\FF,M)$-formulae as a consequence of Theorem \ref{thm:compl forcing thm}.
But then by Lemma \ref{lem:varphi*}, 
$$T=\{\langle\gbr{\varphi},x\rangle\mid\gbr{\varphi}\in\Fm_1,x\in M,p^x\Vdash^\MM_\FF\varphi^*\}\in\C$$
is a first-order truth predicate for $M$. 
\end{proof}

\begin{proof}[Proof of Theorem \ref{theorem:FailureDefLemma}]
 Let $M$ be a countable transitive model of $\ZF^-$. Assume, towards a contradiction, that the set $\Set{\langle p,\sigma,\tau\rangle}{p\Vdash^M_{\FF}\anf{\sigma=\tau}}$ is definable over $M$. Then $\MM=\langle M,\mathrm{Def}(M)\rangle$ is a model of $\GB^-$ and $\FF$ satisfies the definability lemma for atomic formulae over $\MM$. By Theorem \ref{thm:ft truth pred}, there is a first-order truth predicate for $M$ that is first-order definable over $M$. This contradicts Tarski's theorem on the undefinability of truth. 
\end{proof}

Theorem \ref{thm:ft truth pred} can also be used to provide an alternative proof of the following well-known fact.

\begin{corollary}
If $\MM$ is a model of $\KM$, then $\C$ contains a first-order truth predicate for $M$. 
\end{corollary}

\begin{proof}
By Theorem \ref{thm:KM bc}, $\FF$ satisfies the definability lemma for all $\mathcal{L}^n$-formulae over $\MM$, so by Theorem \ref{thm:ft truth pred}, $\C$ contains a first-order truth predicate for $M$. 
\end{proof}

We can even do better and find fixed names $\nu$ and $\mu\in M^\FF$ such that the forcing theorem for
$\nu=\mu$ implies the existence of a first-order truth predicate.

\begin{lemma}\label{lem:names truth pred} 
 There exist $\mu,\nu\in M^\FF$ and $\{\langle\gbr{\varphi},q_\gbr{\varphi}\rangle\mid\gbr{\varphi}\in\Fm_1\}\in M$ such that
\begin{enumerate-(1)}
 \item If $\varphi$ has one free variable and $x\in M$, then $M\models\varphi(x)$ iff for all $r\leq_\FF p^x,q_\gbr{\varphi}$, $r\Vdash_\FF\mu=\nu$.
 \item If $\varphi$ is a sentence, then $M\models\varphi$ iff $q_\gbr{\varphi}\Vdash_\FF\mu=\nu$. 
\end{enumerate-(1)}
\end{lemma}

\begin{proof}Since the proof of (2) is a simplified version of the proof of (1), we only verify (1). 
Choose an antichain $\{q^n\mid n\in\omega\}\subseteq\FF$ such that for every
$n\in\omega$, $0\notin\dom{q^n}$, e.g. take 
$$q^n=\langle\{1,\dots,n+1\},\{\langle1,n+1\rangle\},\emptyset\rangle.$$ 
Consider the names $\nu_\gbr{\varphi^*},\mu_\gbr{\varphi^*}$ as defined in Lemma \ref{lem:names uft}. 
We will only consider non-atomic formulae, since all atomic formulae with at most one free variable 
are either tautologically true or false. The proof of Lemma \ref{lem:names uft}
shows that for $\gbr{\varphi}\in\Fm_1$ with $\varphi$ non-atomic, all elements of $\nu_\gbr{\varphi^*},\mu_\gbr{\varphi^*}$ are of the form $\langle\tau,\one_\FF\rangle$ 
for some $\tau\in M^\FF$. Let $k:\omega\ra\Fm_1$ be a bijection and let $j:\omega\ra\Fm_{\omega_1,0}^{\Vdash}(\FF,M)$ be
given by $j(n)=\gbr{\varphi^*}$, where $k(n)=\gbr{\varphi}$. Now set 
\begin{align*}
 \nu&=\{\langle\tau,q^n\rangle\mid\langle\tau,\one_\FF\rangle\in\nu_{j(n)}\}\\
 \mu&=\{\langle\tau,q^n\rangle\mid\langle\tau,\one_\FF\rangle\in\mu_{j(n)}\}.
\end{align*}
This yields that $q^n\Vdash_\FF\nu=\nu_{j(n)}$ and $q^n\Vdash_\FF\mu=\mu_{j(n)}$ for each $n\in\omega$. Moreover, since $0\notin\dom{q^n}$, 
$p_0^x$ and $q^n$ are compatible for every $x\in M$ and $n\in\omega$. For $\gbr{\varphi}\in\Fm_1$, we put 
$$q_\gbr{\varphi}=q^{k^{-1}(\gbr{\varphi})}.$$
To check (1), suppose first that $M\models\varphi(x)$ for some $\L_\in$-formula $\varphi$ and $x\in M$. 
Let $r\in\FF$ be such that $r\leq_\FF p^x,q_\gbr{\varphi}$. Since $r\leq_\FF q_\gbr{\varphi}$, 
$r\Vdash_\FF\nu=\nu_\gbr{\varphi^*}\wedge\mu=\mu_\gbr{\varphi^*}$. On the other hand, since $r\leq_\FF p^x$,
Lemma \ref{lem:varphi*} implies that $r\Vdash_\FF\varphi^*$, i.e. by Lemma \ref{lem:names uft}, $r\Vdash_\FF\nu=\mu$. 
Conversely, assume that $M\models\neg\varphi(x)$. By (1) applied to the negation of $\varphi$, we have 
that for all $r\leq_\FF p^x,q_\gbr{\neg\varphi}$, $r\Vdash_\FF\mu=\nu$. Since $p^x$ and $q_\gbr{\neg\varphi}$ are 
compatible, such $r$ exists. Now let $\pi$ be the automorphism on $\FF$ which for $p=\langle d_p,e_p,f_p\rangle$
swaps all occurrences of $k^{-1}(\gbr{\neg\varphi})$ and $k^{-1}(\gbr{\varphi})$ in $d_p,e_p$ and $\dom{f_p}$. 
Then $\pi(q_\gbr{\neg\varphi})\leq_\FF q_\gbr{\varphi}$ and $\pi(p^x)=p^x$. 
In particular, $\pi(r)\leq_\FF p^x,q_\gbr{\varphi}$ and so $\pi(r)\Vdash_\FF\nu=\nu_\gbr{\varphi^*}\wedge\mu=\mu_\gbr{\varphi^*}$.
Moreover, since $\pi(r)\leq_\FF p^x$, by Lemma \ref{lem:varphi*} we obtain $\pi(r)\Vdash_\FF\neg\varphi^*$. Finally,\
by Lemma \ref{lem:names uft}, this proves that $\pi(r)\Vdash_\FF\mu_\gbr{\varphi^*}\neq\nu_\gbr{\varphi^*}$ and so
$\pi(r)\Vdash_\FF\mu\neq\nu$. 
\end{proof}

\begin{corollary} 
 There exist $\nu,\mu\in M^\FF$ such that if $\{p\in\FF\mid p\Vdash_\FF\mu=\nu\}\in\C$, then $\C$ contains 
a first-order truth predicate for $M$. In particular, $\{p\in\FF\mid p\Vdash_\FF\mu=\nu\}$ is not definable over $M$. \qed
\end{corollary}



In the remainder of this section, we show that amenability of the forcing relation for the forcing $\FF$ can 
consistently fail. We will work with a countable, transitive Paris model $M\models\ZF^-$ (see Section \ref{section:Intro}).
Note that the least $\alpha$ such that $L_\alpha\models\ZF^-$ is such a model.

\begin{lemma}\label{lem:fr not amenable}
 Let $M$ be a countable transitive Paris model with $M\models\ZF^-$. Then 
$$X=\{q_\gbr{\varphi}\mid\gbr{\varphi}\in\Fm_0,q_\gbr{\varphi}\Vdash_\FF\mu=\nu\}$$
is not an element of $M$, where $q_\gbr{\varphi},\mu,\nu$ are as in Lemma \ref{lem:names truth pred}.
\end{lemma}

\begin{proof}
Suppose for a contradiction that $X\in M$. Observe that for every $\L_\in$-sentence,
\begin{equation}\label{al:x}
 M\models\varphi ~ \Longleftrightarrow ~ q_\gbr{\varphi}\in X. 
\end{equation}
Consider $C=\{\gbr{\varphi}\mid q_\gbr{\exists!x\in\On\,\varphi(x)}\in X\}.$
Since $X\in M$, so is $C$. Observe that we can order the elements of $C$ by
$$\gbr{\varphi}<\gbr{\psi} ~\Llr ~q_\gbr{\exists x,y\in\On[x<y\wedge\varphi(x)\wedge\psi(y)]}\in X.$$
As a consequence of \eqref{al:x}, we know that $\langle C,<\rangle$ has order type $\On^M$, a contradiction.
\end{proof}

In particular, this shows Theorem \ref{theorem:FailureAmenability}.

\section{A failure of the truth lemma}\label{sec:ftl}
In this section, we show that the truth lemma can consistently fail for class forcing. 
Note that by Lemma \ref{lem:names uft}, if we find a notion of class forcing and an infinitary formula for which the truth lemma fails, 
then we automatically obtain that it fails for $\anf{v_0=v_1}$. 

We need the following basic result about two-step iterations of class forcing.

\begin{lemma}[\cite{MR1780138}, Lemma 2.30 (a)]\label{lem:2step}
 Let $\MM=\langle M,\C\rangle$ be a countable transitive model of $\GB^-$, let $\PP$ be a tame notion of class forcing and let $\dot{\QQ}\in\C^\PP$ be a class name for a preorder. Then we define the 
two-step iteration of $\PP$ and $\dot{\QQ}$ by 
$$\PP\ast\dot{\QQ}=\{\langle p,\dot{q}\rangle\mid p\in\PP\wedge p\Vdash_\PP\dot{q}\in\dot{\QQ}\}$$
equipped with the ordering given by 
$\langle p_0,\dot{q_0}\rangle\leq_{\PP\ast\dot{\QQ}}\langle p_1,\dot{q_1}\rangle$ iff $p_0\leq_\PP p_1$ and $p_0\Vdash_\PP \dot{q_0}\leq_{\dot{\QQ}}\dot{q_1}$.
If $G$ is $\PP$-generic and $H$ is $\dot{\QQ}^G$-generic over $M[G]$, then 
$G\ast H=\{\langle p,\dot{q}\rangle\mid p\in G\wedge\dot{q}^G\in H\}$ is $\PP\ast\dot{\QQ}$-generic over $M$. 
\end{lemma}

If $\PP$ is any notion of class forcing that satisfies the forcing theorem, we denote by $\dot\FF$ the canonical class $\PP$-name for $\FF^{M[G]}$ in a $\PP$-generic
extension $M[G]$. As an example of the failure of the truth lemma, we will consider two-step iterations where the second iterand will be of the form $\FF^{M[G]}$,
where $G$ is generic for the first iterand.

\begin{theorem}\label{thm:failure tl}
Let $\MM=\langle M,\C\rangle$ be a countable transitive model of $\GB$. Let $\PP$ be a notion of class forcing which is definable over $M$ and has 
the following properties:
\begin{enumerate-(a)}
 \item $\PP$ is tame. 
 \item There is a $\PP$-generic filter $G$ such that $M[G]$ is a Paris model.
 \item For every $p\in G$ there is a $\PP$-generic filter $\bar{G}$ such that $M[\bar{G}]$ is not a Paris model.
\end{enumerate-(a)}
Then the truth lemma fails for $\PP\ast\dot{\FF}$. 
\end{theorem}

\begin{proof} 
 We will find an infinitary formula $\Phi$ in $\L_{\On,0}^{\Vdash}(\PP\ast\dot\FF,M)$ such that if $G*H$ is $\PP\ast\dot\FF$-generic over $\MM$, then $\Phi^{G*H}$ expresses that $M[G]$ is a Paris model. Using this, we choose $G$ as in (b). Then $\Phi^{G*H}$ holds, while by (c), there cannot be a condition in $G*H$ forcing this, implying that the truth lemma fails for $\PP\ast\dot{\FF}$.

Given a formula $\varphi$ with exactly one free variable, let $\Psi_{\ulcorner\varphi\urcorner}$ denote the formula
$$\varphi(v_0) ~\wedge ~\forall v_1~\left[\varphi(v_1) ~\ra ~ v_1=v_0\right],$$ i.e. $\Psi_{\ulcorner\varphi\urcorner}(x)$ states that ``\emph{$x$ is unique such that $\varphi(x)$ holds}''. Similarly, let $\Omega(x)$ be the formula expressing that $x$ is an ordinal. 
If $G$ is $\PP$-generic over $M$, then $M[G]$ is a Paris model if and only if for all ordinals $\alpha\in M[G]$ there is $\varphi$ such that 
$M[G]\models\Psi_\gbr{\varphi}(\alpha)$. Now recall that as described in Section \ref{sec:fr not def}, for each $\varphi$ we can assign to $\Psi_{\ulcorner\varphi\urcorner}$, $\Omega$ and $n\in\omega$ infinitary formulae $(\Psi_\gbr{\varphi})^{*}_n$ and $\Omega^{*}_n$ in the forcing language of $\FF^{M[G]}$ with the properties (as in Lemma \ref{lem:varphi*})
\begin{align*}M[G]\models\Psi_{\ulcorner\varphi\urcorner}(x)\,&\iff\,p_n^x\Vdash_{\FF^{M[G]}}^{\MM[G]}(\Psi_\gbr{\varphi})^{*}_n\\
M[G]\models\Omega(x)\,&\iff\,p_n^x\Vdash_{\FF^{M[G]}}^{\MM[G]}\Omega^{*}_n.
\end{align*}
However, since we will need infinitary formulae in the forcing language of $\PP\ast\dot\FF$, we have to modify this approach slightly.
For a formula $\psi$ and $n\in\omega$, we define $\psi_n^{**}$ in the same way we defined $\psi_n^*$, but we replace every occurence of 
some condition $p\in\FF^{M[G]}$ by $\langle\one_\PP,\check{p}\rangle\in\PP\ast\dot\FF$. Note that this is possible, 
since for every condition $p$ which appears in $\psi_n^*$, the function $f_p$ is empty, and so $p\in M$. 
Let $$\Phi=\bigwedge_{n\in\omega}\;\bigvee_{\gbr{\varphi}\in\Fm_1}\left[\Omega_n^{**} ~\longrightarrow ~(\Psi_\gbr{\varphi})_n^{**}\right].$$

We claim that $M[G]$ is a Paris model if and only if $M[G][H]\models\Phi^{G\ast H}$ holds for every (or, equivalently, for some) filter $H$ which is $\FF^{M[G]}$-generic over $M[G]$.  
Suppose first that $M[G]$ is a Paris model. Let $H$ be $\FF^{M[G]}$-generic over $M[G]$ and $n\in\omega$. By a density argument, there is $x\in M[G]$ such that $p_n^x$ (as defined in ${M[G]}$) is in $H$. Since $M[G]$ is a Paris model, 
either $M[G]\models\lnot\Omega(x)$ or there is some formula $\varphi$ such that $M[G]\models\Psi_\gbr{\varphi}(x)$. Let $\dot{x}\in M^\PP$ be a name 
for $x$. Since $\PP$ is tame, it satisfies the truth lemma and hence there is $q\in\PP$ with 
$q\Vdash_\PP^\MM[\Omega(\dot x)\to\Psi_\gbr{\varphi}(\dot x)]$. Let $\dot p_n^x$ be a $\PP$-name for $p_n^x\in\FF^{M[G]}$. 
Then $\langle q,\dot p_n^x\rangle\Vdash_{\PP\ast\dot\FF}^\MM[\Omega_n^{**}\to(\Psi_\gbr{\varphi})_n^{**}]$.

Conversely, suppose that $M[G][H]\models\Phi^{G\ast H}$. Let $\alpha\in M[G]$ be an ordinal. Let $n\in\omega$ such that $p_n^\alpha\in H$. By assumption, there is $\gbr{\varphi}\in\Fm_1$ such that $M[G][H]\models((\Psi_\gbr{\varphi})_n^{**})^{G\ast H}$. We want to verify that $M[G]\models\Psi_\gbr{\varphi}(\alpha).$
If not, we can proceed as before and obtain $q\in G$ such that
$\langle q,\dot p_n^\alpha\rangle\Vdash_{\PP\ast\dot\FF}^\MM(\neg\Psi_\gbr{\varphi})_n^{**}$, which is contradictory. 
\end{proof}

\begin{remark}\label{theremark}
By a special case of more general results in \cite{MR2140616} and \cite{MR3087066} there is a tame notion of class forcing $\PP^*$ such that for every countable transitive $\GB^-$-model of the form $\MM=\langle M,\mathrm{Def}(M)\rangle$, there is
a $\PP^*$-generic filter $G$ over $\MM$ such that $M[G]$ is pointwise definable. For the benefit of the reader, we will describe
a very simple tame notion of class forcing $\PP$ and indicate a proof that there is a $\PP$-generic extension which is a Paris model over any countable transitive $\GB^-$-model $\MM=\langle M,\mathrm{Def}(M)\rangle$ such that $\langle M,\in\rangle\models V=L$.
The outline of the argument follows the proof of \cite[Theorem 2.8]{MR2140616}. 

$\PP$ is a two-step iteration, where the first step is to force with $\QQ=\langle 2^{<\On},\supseteq \rangle$ (note that this forcing notion does not add new sets) and construct a 
$\QQ$-generic filter $U$ such that all ordinals of $M$ are first-order definable over $\langle M,\in,U\rangle$ without parameters (this construction of $U$ is done as in \cite{MR3087066}). 
The second step is to code the generic $U$ into the continuum function, using a reverse Easton iteration. 
Since the ground model $\mathsf L$ is definable in the extension and every element of $\mathsf L$ is definable in $\mathsf L$ from an ordinal, it follows that $M[G]$ is pointwise definable. 
\end{remark}

The forcing notion $\PP$ described in Remark \ref{theremark} satisfies (a) and (b) over any countable 
transitive model of $\GB^-$ of the form $\langle M,\mathrm{Def}(M)\rangle$. 
We will now give two consistent examples of such models over 
which $\PP$ also satisfies Condition (c) in the statement of Theorem \ref{thm:failure tl}.

\begin{example}
\begin{enumerate-(1)}
 \item The simplest possibility is to start with a model 
$\MM=\langle M,\C\rangle$ of $\KM$. By forcing over the first-order part $\langle M,\mathrm{Def}(M)\rangle$ of $\MM$, we may obtain a $\PP$-generic filter $G$
such that $M[G]$ is a Paris model. On the other hand, we can force over the $\KM$-model $\MM$ and choose 
a filter $\bar G$ which is $\PP$-generic over $\MM$. Since $\PP$ is tame, by \cite[Theorem 23]{antos2015class} $\MM[\bar G]\models\KM$. 
But this is a contradiction, since no model of $\KM$ is a Paris model:

 Suppose for a contradiction that $\mathbb N=\langle N,\mathcal D\rangle$ is a model of $\KM$ which is a Paris model. 
 Using that $\mathcal D$ contains a first-order truth predicate for formulae with one free variable, 
 it follows that $\mathcal D$ contains a surjection from $\omega$ to $\On^N$, contradicting Replacement.
\item If we want to avoid $\KM$, we can instead start with a countable transitive model $\langle M,\in\rangle$ of $\ZFC$ which has cardinality $\aleph_1$ in $\mathsf{L}[M]$ and which is closed under countable sequences in $\mathsf{L}[M]$. 
Now since the forcing $\PP$ is $\sigma$-closed in $M$, for every $p\in\PP^M$ there is a $\PP^M$-generic filter $G_p$ over $M$ in $\mathsf{L}[M]$ containing $p$.  
Since $M$ is uncountable in $\mathsf{L}[M]$, no generic extension of the form $M[G_p]$ is a Paris model. 
Note that it is easy to obtain such a model $M$ starting in a model of $\mathsf{V}=\mathsf{L}$ with an inaccessible cardinal and then forcing with $\Col(\omega,\omega_1)$. 
\end{enumerate-(1)} 
\end{example}

In particular, this proves Theorem \ref{theorem:FailureTruthLemma}.

\section{Non-isomorphic Boolean completions}

Any separative set-sized partial order $\PP$ has a Boolean completion that is
provided by the regular open subsets of $\PP$, and this completion is unique: if $\BB_0$ and $\BB_1$ are 
both Boolean completions of $\PP$ and $e_0:\PP\ra\BB_0$ and $e_1:\PP\ra\BB_1$ are dense embeddings, then 
one can define an isomorphism by $f(b)=\sup\{e_1(p)\mid p\in\PP\wedge e_0(p)\leq b\}$ for $b\in\BB_0$. Moreover,
$f$ fixes $\PP$ in the sense that $f(e_0(p))=e_1(p)$ for every condition $p\in\PP$. 
This proof however does not work for class forcing since in general, we can only form set-sized suprema within $\MM$-complete Boolean algebras.
In the remainder of this section, we will show that the result actually fails for class forcing.

\begin{definition}\label{def:unique bc}
We say that a notion of class forcing \emph{$\PP$ has a unique Boolean completion in $\MM$}, if $\PP$ has a Boolean completion 
$\BB_0$ in $\MM$ and for every other Boolean completion $\BB_1$ of $\PP$ in $\MM$ 
there is an isomorphism in $\V$ between $\BB_0$ and $\BB_1$ which fixes $\PP$. 
\end{definition}

\begin{definition}
We say that a notion of class forcing $\PP$ satisfies the \emph{$\On$-chain condition} (or simply \emph{$\On$-cc}) over $\MM$,
if every antichain of $\PP$ which is in $\C$ is already in $M$.  
\end{definition}

\begin{lemma}\label{lemma:add sup}
 Let $\PP$ be a notion of class forcing for $\MM$ which satisfies the forcing theorem and let $A\subseteq\PP$ be a class in $\C$ such that 
 $A$ does not have a supremum in $\PP$. Let $\QQ=\PP\cup\{\sup A\}$, where $\sup A\in M$ is a new element which is not in $\PP$ and let $\le_{\QQ}$ be such that
 \begin{align*}
 \le_{\QQ}&\upharpoonright(\PP\times\PP)=\,\le_{\PP},\\
 \sup A\leq_\QQ p&\Llr\forall a\in A(a\leq_\QQ p)\textrm{ and}\\
 p\leq_\QQ\sup A&\Llr A\text{ is predense below }p.
 \end{align*}
Then $\QQ$ satisfies the forcing theorem. 
\end{lemma}
\begin{proof}
For $\sigma\in M^\QQ$, let $\sigma^+$ denote the $\PP$-name obtained from $\sigma$ by replacing every occurrence of $\sup A$ in $\tc(\sigma)$
by $\one_\PP$, and let $\sigma^-\in M^\PP$ be the name obtained from $\sigma$ by removing every pair of the form $\langle\mu,\sup A\rangle$
from $\tc(\sigma)$. To be precise about the latter, we inductively define $\sigma^-=\{\langle\tau^-,p\rangle\in\sigma\mid p\ne\sup A\}$. One easily checks that for all $p\in\QQ$ and all $\QQ$-names $\sigma$ and $\tau$,
$$p\Vdash_\QQ\sigma\in\tau\Llr\forall q\in\PP\,\left[q\leq_\QQ p\ra[(q\bot_\PP A\ra q\Vdash_\PP\sigma^-\in\tau^-)\wedge(q\leq_\PP A\ra q\Vdash_\PP\sigma^+\in\tau^+)]\right],$$
where $q\le_\PP A$ abbreviates $\exists a\in A\ (q\le_\PP a)$ and $q\bot_\PP A$ abbreviates $\forall a\in A\ (q\bot_\PP a)$. 
\end{proof}

Under below assumptions on $\MM$, the authors originally showed that there is a particular class forcing for $\MM$ which has a non-unique Boolean completion. The following more general result was then observed and pointed out to the authors by Joel David Hamkins.

\begin{theorem}
Let $\MM=\langle M,\C\rangle$ be a model of $\GB^-$ such that $\C$ contains a well-order of $M$ of order type $\On^M$. Then a separative antisymmetric notion of class forcing $\PP$ for $\MM$
has a unique Boolean completion in $\MM$ if and only if it satisfies the $\On$-cc over $\MM$.  
\end{theorem}

\begin{proof}
Suppose first that $\PP$ satisfies the $\On$-cc over $\MM$. As can easily be observed from the combinatorial characterization of pretameness in \cite[Chapter 2.2]{MR1780138}, since $\C$ contains a global well-order of $M$, $\PP$ is pretame and 
therefore it has a Boolean completion $\BB(\PP)\in\C$ by Corollary \ref{cor:pretame bc}. Let $\BB$ be another Boolean completion of $\PP$.
Without loss of generality, we can assume that $\PP$ is a subset of $\BB$. Then every element $b\in\BB$ satisfies 
$b=\sup_\BB D_b$, where $D_b=\{p\in\PP\mid p\leq_{\BB} b\}\in\C$. 
But using the global well-order of $M$ we obtain that $D_b$ contains an antichain which is maximal in $D_b$. Moreover, since $\PP$ satisfies the $\On$-cc,
every such antichain lies in $M$. Furthermore, observe that if $A$ and $A'$ are two antichains which are maximal in $D_b$, then 
$\sup_\BB A=\sup_\BB A'=b$ and $\sup_{\BB(\PP)}A=\sup_{\BB(\PP)}A'$. But this gives a canonical embedding of 
$\BB$ into $\BB(\PP)$ which fixes $\PP$. It is clearly surjective, since the same argument as above can be done within $\BB(\PP)$.
This shows that $\PP$ has a unique Boolean completion in the sense of Definition \ref{def:unique bc}.

Conversely, assume that $\C$ contains an antichain $A$ of $\PP$ which is not in $M$. Moreover, suppose that $\PP$ satisfies the 
forcing theorem and therefore it has a Boolean completion $\BB(\PP)$ as provided by Theorem \ref{thm:compl forcing thm} (otherwise 
$\PP$ has no Boolean completion and in particular no unique Boolean completion). We may assume that $\PP\subseteq\BB(\PP)$. Note that $A$ 
remains a class-sized antichain in $\BB(\PP)$. We may assume, by replacing $A$ with some subclass of $A$ if necessary, that $A$ has order type $\On^M$ with respect to the given wellorder of $M$.

We claim that $A$ contains a subclass in $\C$ which does not have a supremum in $\BB(\PP)$. 
Assume, towards a contradiction, that every subclass of $A$ which is in $\C$ has a supremum in $\BB(\PP)$. 
Then since $A$ is an antichain, the mapping $i:\mathcal P(A)\cap\C\ra\BB(\PP)$ given by $i(X)=\sup_{\BB(\PP)} X$ is injective. 
By assumption, there is a surjection $\On^M\ra M$. But by the above, this yields a surjection from $\On^M$ onto $\mathcal P(A)\cap\C$ and 
hence there is a surjection $s:\On^M\ra\mathcal P(\On^M)\cap\C$, since $A$ has order type $\On^M$. Consider now 
$C=\{\alpha\in\On^M\mid\alpha\notin s(\alpha)\}$ and let $\alpha\in\On^M$ such that $C=s(\alpha)$. But this leads to Russell's paradox, since 
$\alpha\in s(\alpha)=C$ if and only if $\alpha\notin s(\alpha)$.

We may therefore, without loss of generality, assume that $A$ does not have a supremum in $\BB(\PP)$. Using Lemma \ref{lemma:add sup},
we can enlarge $\BB(\PP)$ to a notion of forcing $\QQ$ which satisfies the forcing theorem, such that $A$ has a supremum in $\QQ$. 
Moreover, it is straightforward to check that this notion of forcing $\QQ$ provided by Lemma \ref{lemma:add sup} is antisymmetric and separative, so we can find a Boolean completion $\BB(\QQ)$ of $\QQ$ as provided by Theorem \ref{thm:compl forcing thm}. We may assume that $\QQ\subseteq\BB(\QQ)$. Since $\PP$ is dense in $\QQ$, $\BB(\QQ)$ is also a Boolean completion of $\PP$. 

It is now straightforward to check that any isomorphism between $\BB(\PP)$ and $\BB(\QQ)$ which fixes $\PP$ will also fix $\BB(\PP)$. 
Since $\BB(\PP)\subsetneq\QQ\subseteq\BB(\QQ)$, such an isomorphism cannot exist.
\end{proof}

\begin{example}
There are many examples of notions of forcing $\PP$ over a model $\MM$ of $\GB^-$ with a global well-order of order type $\On^M$
which have a Boolean completion which is not unique up to isomorphisms preserving $\PP$.
The simplest examples are the collapse forcings $\Col(\omega,\On)^M$ and $\Col_*(\omega,\On)^M$ which both have
a Boolean completion by Lemma \ref{lem:col proj} and Theorem \ref{thm:compl forcing thm}. However, they clearly have 
class-sized antichains such as $\{\langle0,\alpha\rangle\mid\alpha\in\On^M\}$ and therefore do not have a unique Boolean completion
by the characterization above. 
\end{example}

\section{Open Questions} 
A topic which has not yet been discussed systematically is the preservation of the axioms of $\ZFC$ in class forcing extensions with the generic filter and the ground model as predicates.
Preservation of Replacement and Power Set have been characterized by Sy Friedman in \cite{MR1780138}. 
However, as one can easily observe, all collapse forcings which we have mentioned in this paper do not only destroy
Replacement but also Separation. It is therefore natural to ask the following question, which shall be discussed in \cite{class_forcing2}.

\begin{question}
Is there a notion of class forcing which preserves Separation, but not Replacement? If so, does preservation of Separation
already imply the forcing theorem? 
\end{question}

Since the notions of forcing used in the above counterexamples destroy many axioms of $\ZFC$ even without the generic filter and the ground model as predicates, the following question arises. 

\begin{question} 
Suppose that $\mathbb{P}$ is a class forcing over a countable transitive model $M$ of $\mathsf{ZFC}$ such that $(M[G],\in)$ (without the generic filter and the ground model as predicates) satisfies $\mathsf{ZFC}$ for every $\mathbb{P}$-generic filter $G$ over $M$. Does $\mathbb{P}$ satisfy the forcing theorem? 
\end{question} 

The reason why the forcing $\FF^M$, as introduced in Section \ref{examples}, does not satisfy the forcing theorem, is essentially because it adds a specific subset $E$ of $\omega$. It would be interesting to know how closely connected failures of the forcing theorem are to such \emph{witnessing sets}:

\begin{question}
  Let $M$ be a countable transitive model of $\mathsf{ZFC}$. Assume $\PP$ and $\QQ$ are both notions of class forcing for $M$, which have the same generic extensions without predicates, i.e.\ for which $\{M[G]\mid G$ is generic for $\PP$ over $M\}=\{M[H]\mid H$ is generic for $\QQ$ over $M\}$. Assume further that $\PP$ satisfies the forcing theorem. Does it follow that $\QQ$ satisfies the forcing theorem?
\end{question}

A further question is concerned with the amenability of 
the forcing relation. In Section \ref{sec:fr not def}, we have shown that for 
countable transitive Paris models $M\models\ZF^-$, the forcing relation for $\FF^M$ 
is not amenable. The following questions remain open.

\begin{question}\label{question:amenability}
\begin{enumerate} 
\item Is there a notion of class forcing $\mathbb{P}$ such that for every model $M$ of $\ZFC$, the forcing relation for $\mathbb{P}$ is not amenable for $M$? 
\item Are there any provable implications between the amenability of the forcing relation and the truth lemma? 
\end{enumerate} 
\end{question}

We have shown that in certain models, there is a tame notion of class forcing $\PP$ such that $\PP*\dot\FF$ does not satisfy the truth lemma. However the following question is still open.

\begin{question}
Is there a notion of class forcing $\PP$ such that for every model of $\ZFC$, the truth lemma for $\PP$ fails?
\end{question}

Furthermore, we do not know whether the two-step iteration of notions of class forcing satisfying the truth lemma, with the first forcing in fact being tame, again satisfies the truth lemma, thus also the following question is open.

\begin{question} 
 Does the forcing $\FF$, as introduced in Section \ref{examples}, provably satisfy the Truth Lemma? 
\end{question}


\nocite{*} 

\bibliographystyle{alpha}
 \bibliography{class-forcing}


\end{document}